\documentclass[a4paper,bibliography=totoc]{scrartcl}
\usepackage[utf8]{inputenc}
\usepackage{amsmath,amssymb,amsthm,amsfonts}
\usepackage{mathtools,thmtools,thm-restate}
\usepackage[british]{babel}
\usepackage{csquotes}
\usepackage{hyperref}
\usepackage[nameinlink,noabbrev]{cleveref}
\usepackage{enumerate}
\usepackage{graphicx}
\usepackage{xcolor}

\usepackage[maxbibnames=5,sorting=nyt,sortcites,giveninits=true]{biblatex}
\DeclareNameAlias{sortname}{family-given}
\title{Space-time divergence lemmas and optimal non-reversible lifts of diffusions on Riemannian manifolds with boundary}
\author{Andreas Eberle\thanks{E-Mail: \href{mailto:eberle@uni-bonn.de}{eberle@uni-bonn.de}, ORCID: \href{https://orcid.org/0000-0003-0346-3820}{0000-0003-0346-3820}}\qquad Francis Lörler\thanks{E-Mail: \href{mailto:loerler@uni-bonn.de}{loerler@uni-bonn.de}, ORCID: \href{https://orcid.org/0009-0007-3177-1093}{0009-0007-3177-1093}}\medskip\\Institute for Applied Mathematics, University of Bonn.\\ Endenicher Allee 60, 53115 Bonn, Germany.}

\newcommand{\R}     {\mathbb{R}}

\newcommand{\N}     {\mathbb{N}}

\newcommand{\T}     {\mathrm{T}}
\newcommand{\Ncal}  {\mathrm{N}}
\renewcommand{\H}   {\mathbb{H}}
\newcommand{\X}     {\mathfrak{X}}
\newcommand{\D}     {\mathcal{D}}

\renewcommand{\phi} {\varphi}

\newcommand{\Ecal}  {\mathcal{E}}
\newcommand{\diff}  {\mathop{}\!\mathrm{d}}

\newcommand{\muq}   {\overline\mu}
\newcommand{\Mq}    {{\overline{M}}}
\newcommand{\dM}    {{\partial M}}
\newcommand{\nuM}   {\nu_M}
\newcommand{\nudM}  {\nu_\dM}
\newcommand{\nablad}{\nabla^\dM}
\newcommand{\Deltad}{\Delta^\dM}
\newcommand{\Ld}    {L^\dM}
\newcommand{\nablaq}{\overline{\nabla}}
\newcommand{\mudM}  {\mu_\dM}
\newcommand{\muqk}  {\lambda\otimes\hat\mu}

\newcommand{\DC}    {\mathrm{D}}
\newcommand{\sff}   {\mathrm{I\!I}}
\newcommand{\rel}   {\mathrm{rel}}
\newcommand{\RHMC}  {\textup{RHMC}}
\newcommand{\LD}    {\textup{LD}}
\renewcommand{\c}     {\textup{c}}

\let\Re\relax

\DeclareMathOperator{\Re}   {Re}

\DeclareMathOperator{\spn}  {span}
\DeclareMathOperator{\tr}   {tr}
\DeclareMathOperator{\gap}  {gap}
\DeclareMathOperator{\spec} {spec}

\DeclareMathOperator{\dom}  {Dom}

\DeclareMathOperator{\divg} {div}

\DeclareMathOperator{\Ric}  {Ric}
\DeclareMathOperator{\supp} {supp}

\DeclarePairedDelimiterX{\norm}[1]{\lVert}{\rVert}{#1}

\theoremstyle{plain}
\newtheorem{theo}{Theorem}
\newtheorem{lemm}[theo]{Lemma}
\newtheorem{coro}[theo]{Corollary}

\theoremstyle{definition}

\newtheorem{rema}[theo]{Remark}
\newtheorem{rema*}{Remark}
\newtheorem{assu}{Assumption}
\crefname{lemm}{lemma}{lemmas}
\crefname{theo}{theorem}{theorems}
\crefname{assu}{assumption}{assumptions}

\begin{document}

\maketitle

\begin{abstract}\vspace{-\baselineskip}
    Non-reversible lifts reduce the relaxation time of reversible diffusions at most by a square root. For reversible diffusions on domains in Euclidean space, or, more generally, on a Riemannian manifold with boundary, non-reversible lifts are in particular given by the Hamiltonian flow on the tangent bundle, interspersed with random velocity refreshments, or perturbed by Ornstein-Uhlenbeck noise, and reflected at the boundary. In order to prove that for certain choices of parameters, these lifts achieve the optimal square-root reduction up to a constant factor, precise upper bounds on relaxation times are required. A key tool for deriving such bounds by space-time Poincaré inequalities is a quantitative space-time divergence lemma. Extending previous work of Cao, Lu and Wang, we establish such a divergence lemma with explicit constants for general locally convex domains with smooth boundary in Riemannian manifolds satisfying a lower, not necessarily positive, curvature bound. As a consequence, we prove optimality of the lifts described above up to a constant factor, provided the deterministic transport part of the dynamics and the noise are adequately balanced. Our results show for example that an integrated Ornstein-Uhlenbeck process on a locally convex domain with diameter $d$ achieves a relaxation time of the order $d$, whereas, in general, the Poincaré constant of the domain is of the order $d^2$.

    \begin{samepage}
    \par\vspace\baselineskip
    \end{samepage}
\end{abstract}

\section{Introduction}

Long-time convergence to equilibrium of strongly continuous contraction semigroups plays an important role in many areas of probability and analysis, ranging from convergence of Markov processes with applications in sampling to kinetic equations, ergodic theory, and non-equilibrium statistical physics. The degenerate case in which the generator of the semigroup is not an elliptic but only hypoelliptic operator has attracted much attention in the past years, and the study of convergence of such dynamics is known as hypocoercivity \cite{Villani2009Hypocoercivity}. 
Recently, a variational approach based on space-time Poincaré inequalities pioneered by Albritton, Armstrong, Mourrat and Novack \cite{Albritton2021Variational} has proved successful in obtaining sharp and quantitative bounds on the rates of convergence. This approach has further been developed and applied in various scenarios by Cao, Lu and Wang \cite{Cao2019Langevin,Lu2022PDMP}. Furthermore, by casting it in a framework of second-order lifts \cite{EberleLoerler2024Lifts}, the authors have simplified the approach, and an associated lower bound on the relaxation time of lifts has raised the question of existence of optimal lifts achieving this lower bound. 

The goal of this paper is to prove upper bounds of the optimal order on the relaxation time of Langevin dynamics and randomised Hamiltonian Monte Carlo with appropriately chosen parameters on locally convex Riemannian manifolds with boundary that satisfy a lower curvature bound. To this end, we provide a quantitative space-time divergence lemma, which is the key ingredient in the proof of space-time Poincaré inequalities. By considering Riemannian manifolds with boundary as spatial domains, our result allows for the treatment of non-reversible, degenerate dynamics on Riemannian manifolds, as well as dynamics with reflective boundary behaviour. In particular, it shows that, up to a constant factor, the processes described above are optimal lifts of an overdamped Langevin diffusion with reflection at the boundary, provided the deterministic transport part of the dynamics and the noise are adequately balanced.

Probability distributions supported on Riemannian manifolds arise in many applications; through constrained dynamics in physics \cite{lelievre2012langevin,Lelievre2010freeenergy}, as natural parameter spaces in statistics \cite{Chikuse2012Statistics}, in Bayesian inference \cite{RudolfSprungk2023Sphere,Diaconis2013Sampling}, through introduction of an artificial Riemannian geometry on $\R^n$ \cite{LeeVempala2022ManifoldJoys}, or in volume computation of convex bodies \cite{LeeVempala2018Convergence,Gatmiry2024Sampling}, to name a few. Sampling from such distributions is a challenging task, and theoretical convergence bounds for many proposed algorithms \cite{ByrneGirolami2013GeodesicSampling,LeeShenVempala2023Riemannian,ChevallierHMCreflections,GirolamiCalderhead2011Riemann,AfsharDomke2015ReflectionHMC} are still lacking. Our quantitative convergence bounds provide a first step towards the theoretical understanding of sampling algorithms based on discretisations of continuous-time processes on Riemannian manifolds. 
Similarly, dynamics constrained to a convex domain arise in sampling problems with restrictions \cite{AhnChewi2021MirrorLangevin,LanKang2023Constrained}.
The relaxation time of reflected Brownian motion on a convex domain with diameter $d$ is the Poincaré constant of the uniform distribution on this domain, which is of the order $d^2$ in general \cite{Li2012Geometric}.
Our result shows that critically damped Langevin dynamics, which reduces to an integrated Ornstein-Uhlenbeck process reflected at the boundary, and randomised Hamiltonian Monte Carlo with appropriately chosen refresh rate, which reduces to a billiards process with velocity refreshment, achieve a relaxation time of order $d$.

\medskip
Let us first give a brief review of the concept of second-order lifts of reversible diffusions introduced in \cite{EberleLoerler2024Lifts}. Consider a time-homogeneous Markov process $(X_t,V_t)_{t\geq 0}$ with invariant probability measure $\hat\mu$ on the tangent bundle $\T M$ of a Riemannian manifold $(M,g)$ with boundary and let $\pi\colon\T M\to M$ be the natural projection. Furthermore, let $(Z_t)_{t\geq 0}$ be a reversible diffusion process on $M$ with invariant probability measure $\mu = \hat\mu\circ\pi^{-1}$. Their associated transition semigroups acting on $L^2(\hat\mu)$ and $L^2(\mu)$ are $(\hat P_t)_{t\geq 0}$ and $(P_t)_{t\geq 0}$, and their infinitesimal generators are denoted by $(\hat L,\dom(\hat L))$ and $(L,\dom(L))$, respectively. Then $(\hat P_t)_{t\geq 0}$ is a \emph{second-order lift} of $(P_t)_{t\geq 0}$ if 
\begin{equation}\label{eq:deflift0}
    \dom(L)\subseteq\{f\in L^2(\mu)\colon f\circ\pi\in\dom(\hat L)\}\,,
\end{equation}
and for all $f,g\in\dom(L)$ we have
\begin{equation}\label{eq:deflift1}
    \int_{\T M}\hat L(f\circ \pi)(g\circ\pi)\diff\hat\mu=0
\end{equation}
and
\begin{equation}\label{eq:deflift2}
    \frac{1}{2}\int_{\T M}\hat L(f\circ\pi)\hat L(g\circ\pi)\diff\hat\mu =\Ecal(f,g)\,,
\end{equation}
where $\Ecal(f,g)=-\int_M fLg\diff\mu$ is the Dirichlet form associated to $(L,\dom(L))$. 

Since the process $(Z_t)_{t\geq 0}$ is assumed to be reversible, its generator is self-adjoint, and the associated semigroup satisfies
\begin{equation}\label{eq:Ptnorm}
    \norm{P_t}_{L_0^2(\mu)\to L_0^2(\mu)} = \exp(-\lambda t)
\end{equation}
with the decay rate $\lambda$ given by the spectral gap 
\begin{equation}\label{eq:gap}
    \gap(L)=\inf\{\Re(\alpha)\colon\alpha\in\spec(-L|_{L_0^2(\mu)})\}
\end{equation}
of $(L,\dom(L))$ in the Hilbert space $L^2(\mu)$. Here $L_0^2(\mu) = \{f\in L^2(\mu)\colon\int_Mf\diff\mu=0\}$ denotes the subspace of $L^2(\mu)$ of mean-zero functions. In contrast, the strong generator of a \emph{non-reversible} Markov process is no longer self-adjoint and the operator norm of the associated semigroup is no longer a pure exponential. In particular, the spectral gap only coincides with the asymptotic decay rate of the semigroup under additional assumptions ensuring a spectral mapping theorem relating the spectrum of $P_t$ to that of $L$, see \cite{engel1999semigroups}. This motivates the introduction of the non-asymptotic relaxation time 
\begin{equation*}
    t_\rel(\hat P) = \inf\{t\geq0\colon\norm{\hat P_tf}_{L^2(\mu)}\leq e^{-1}\norm{f}_{L^2(\mu)}\textup{ for all }f\in L_0^2(\mu)\}
\end{equation*}
of $(\hat P_t)_{t\geq 0}$. For reversible processes, \eqref{eq:Ptnorm} shows that the relaxation time coincides with the usual definition as the inverse of the spectral gap. A crucial consequence of the lift property is the lower bound
\begin{equation*}
    t_\rel(\hat P) \geq \frac{1}{2\sqrt{2}}\sqrt{t_\rel(P)}
\end{equation*}
on the relaxation time of an arbitrary lift, see \cite[Theorem 11]{EberleLoerler2024Lifts}, which shows that convergence measured by the relaxation time can at most be accelerated by a square root through lifting. A lift $(\hat P_t)_{t\geq 0}$ is hence called \emph{$C$-optimal} if 
\begin{equation}\label{eq:Copt}
    t_\rel(\hat P) \leq \frac{C}{2\sqrt{2}}\sqrt{t_\rel(P)}\,,
\end{equation}
i.e.\ if it achieves this lower bound up to a constant factor $C>0$.
\medskip

Proving $C$-optimality of a lift requires precise upper bounds on the relaxation time. The approach recently introduced by Albritton, Armstrong, Mourrat and Novack \cite{Albritton2021Variational} relies on space-time Poincaré inequalities in order to obtain exponential decay
\begin{equation*}
    \frac{1}{T}\int_t^{t+T}\norm{\hat P_sf}_{L^2(\hat\mu)}^2\diff s \leq e^{-\nu t}\frac{1}{T}\int_0^T\norm{\hat P_s f}_{L^2(\hat\mu)}^2\diff s\qquad\textup{for all }f\in L_0^2(\hat\mu)\textup{ and }t\geq 0
\end{equation*}
with rate $\nu$ of time-averages over some fixed length $T>0$, which yields $t_\rel(\hat P)\leq\frac{1}{\nu}+T$. The key tool in proving a space-time Poincaré inequality is a space-time divergence lemma which can be formulated in terms of the generator $(L,\dom(L))$ of the reversible process alone. Namely, letting $\lambda$ denote the Lebesgue measure on $[0,T]$, any $f\in L_0^2(\lambda\otimes\mu)$ can be written as
\begin{equation*}
    f = \partial_th-Lg
\end{equation*}
with functions $h\in H^{1,2}(\lambda\otimes\mu)$ and $g\in H^{2,2}(\lambda\otimes\mu)$ satisfying Dirichlet boundary conditions in time, as well as $g(t,\cdot)\in\dom(L)$ for $\lambda$-almost every $t\in [0,T]$, together with the regularity estimates
\begin{eqnarray*}
    \norm{h}_{L^2(\muq)}^2+\norm{\nabla g}_{L^2(\muq)}^2 &\leq& c_0(T)\,\norm{f}_{L^2(\muq)}^2\, ,\\
    \norm{\nablaq h}_{L^2(\muq)}^2+\norm{\nablaq\nabla g}_{L^2(\muq)}^2 &\leq& c_1(T)\,\norm{f}_{L^2(\muq)}^2\,
\end{eqnarray*}
on $h$ and $g$, where $\nablaq=(\partial_t,\,\nabla)$ denotes the space-time gradient. 
One goal of this work is to provide solutions to these space-time divergence equations with quantitative expressions for $c_0(T)$ and $c_1(T)$ in the case where $\mu$ has a density $\exp(-U)$ with respect to the Riemannian volume measure and $(Z_t)_{t\geq 0}$ is a Riemannian overdamped Langevin diffusion with generator
\begin{equation*}
    L = \Delta-\langle\nabla U,\nabla\rangle
\end{equation*}
with Neumann boundary conditions. In this case, $f$ can be written as the space-time divergence of the vector field with components $-h$ and $\nabla g$, which motivated the term `divergence lemma' in the literature.
This allows for the treatment of second-order lifts of Riemannian overdamped Langevin dynamics, which include Riemannian randomised Hamiltonian Monte Carlo and Langevin dynamics with reflective boundary behaviour. 

\medskip

We begin in the following section by formally introducing the setting and summarising the Riemannian geometry involved. In particular, we provide an extension of the Reilly formula which arises through an integration-by-parts of the Bochner formula for $L$ and a careful treatment of the boundary terms. Together with a lower curvature bound, it allows to upper bound the Hessian of smooth functions on $M$. In \Cref{sec:divergence}, we prove the divergence lemma with explicit quantitative bounds on the constructed functions. By considering a splitting of $L^2(\mu)$ into high and low modes associated to a spectral decomposition of $L$, as well as into symmetric and antisymmetric functions, we obtain constants of the optimal order conjectured in \cite{Cao2019Langevin}. 
We demonstrate an application of the divergence lemma in \Cref{sec:optimallifts} by showing that, with critical choice of parameters, Langevin dynamics and randomised Hamiltonian Monte Carlo on a Riemannian manifold with boundary satisfying a lower curvature bound are optimal lifts of the corresponding overdamped Langevin diffusion. A brief overview of Riemannian geometry on manifolds with boundary is given in \Cref{sec:appendix}.

\section{Preliminaries and the generalised Reilly formula}\label{sec:setting}

Consider a complete connected and oriented Riemannian manifold $(M,g)$ with (possibly empty) boundary. In particular, the boundary $\dM$ of $M$ is a Riemannian submanifold without boundary, and for simplicity we again denote the induced metric on $\dM$ by $g$. We equip the boundary with the outward unit-normal vector field $N$, and for $x\in M$ and $v,w\in \T_xM$ we write $\langle v,w\rangle_x$ for $g_x(v,w)$. For a short overview over manifolds with boundary, see \Cref{ssec:ManifoldBdry}. 
We equip $M$ with a probability measure $\mu$ absolutely continuous with respect to the Riemannian volume measure $\nu_M$, i.e.\
\begin{equation*}
    \mu(\diff x)=\exp(-U(x))\,\nuM(\diff x)
\end{equation*}
for some function $U\in C^2(M)$ with $\int_M\exp(-U)\diff\nuM = 1$.

Let $\Delta\colon C^\infty(M)\to C^\infty(M),\,u\mapsto\divg(\nabla u)$, where $\nabla$ is the Riemannian gradient and $\divg$ the Riemannian divergence, be the Laplace-Beltrami operator. Then the weighted Laplace-Beltrami operator $L$ is defined as
\begin{equation*}
    Lu = \Delta u-\langle\nabla U,\nabla u\rangle\qquad\textup{for all }u\in C^\infty(M)\,.
\end{equation*}
The operator $L$ satisfies the generalised Bochner-Lichnerowicz-Weitzenböck formula
\begin{equation}\label{eq:bochner2}
    \frac{1}{2}L|\nabla u|^2 = |\nabla^2u|^2 + \langle \nabla u,\nabla L u\rangle + (\Ric+\nabla^2U)(\nabla u,\nabla u)
\end{equation}
for all $u\in C^\infty(M)$, see \cite[Proposition 3]{BakryEmery1985Diffusions}, where $\nabla^2u$ is the Hessian, or second covariant derivative, of $u$, and $\Ric$ is the Ricci curvature tensor of $M$. Letting $\partial_i$ denote the $i$-th coordinate vector field, $|\nabla^2u|^2=\sum_{i,j=1}^d(\nabla^2_{\partial_i,\partial_j}u)^2$ is the squared Frobenius norm of the tensor $\nabla^2u$. Note that \eqref{eq:bochner2} reduces to the usual Bochner identity in case $U$ is constant and thus $L=\Delta$, see e.g.\ \cite{Jost2013Riemannian}.
Let 
\begin{equation*}
    C_\c^\infty(M) = \{u\in C^\infty(M)\colon \supp(u) \textup{ is compact in } M\}
\end{equation*}
denote the set of smooth, compactly supported functions on $M$, where elements of $C_\c^\infty(M)$ are \emph{not} required to vanish at the boundary $\dM$.
Integrating the Bochner identity by parts with respect to the Riemannian volume $\nu_M$ yields the Reilly formula \cite{Reilly1977Formula}
\begin{align*}
    \MoveEqLeft\int_M\left(|\nabla^2u|^2-(\Delta u)^2+\Ric(\nabla u,\nabla u)\right)\diff\nuM\\
    &=\int_\dM\left(H\left(\partial_Nu\right)^2 - 2(\partial_Nu)\Deltad u + h\big(\nablad u,\nablad u\big)\right)\diff\nudM\,,
\end{align*}
for all $u\in C_\c^\infty(M)$, see also \cite[Lemma 7.10]{ColdingMinicozzi2011Minimal} for a proof. Here $h$ is the scalar second fundamental form, $H$ is the mean curvature of the boundary, see \Cref{ssec:HessianSFF}, and $\partial_Nu = \langle\nabla u,N\rangle$. We use the notation $\nablad$ and $\Deltad$ for the Riemannian gradient and the Laplace-Beltrami operator on the boundary manifold $(\dM,g)$, respectively, and the induced Riemannian volume measure on $\dM$ is denoted by $\nudM$.
Denoting the induced measure on the boundary by
\begin{align*}
    \mudM(\diff x)=\exp(-U(x))\,\nudM(\diff x)\,,
\end{align*}
the operator $L$ satisfies the integration-by-parts identity
\begin{align}\label{eq:intbypartsL}
    \int_M \phi L\psi\diff\mu = \int_\dM \phi \langle\nabla\psi,N\rangle\diff\mu_\dM - \int_M\langle\nabla\phi,\nabla\psi\rangle\diff\mu\,
\end{align}
for all smooth, compactly supported functions $\phi,\psi\in C_\c^\infty(M)$. Note that $\mudM$ is not necessarily a probability measure.
Now integrating \eqref{eq:bochner2} by parts with respect to $\mu$ yields the following generalisation of the Reilly formula due to Ma and Du \cite{MaDu2010Extension}, which reduces to the usual Reilly formula in case $U$ is constant. For completeness and the convenience of the reader, we include a proof in \Cref{ssec:BochnerReilly}.
\begin{restatable}[Generalised Reilly formula]{lemm}{GenReilly}\label{lem:GenReilly}
    For any $u\in C_\c^\infty(M)$,
    \begin{align*}
        \MoveEqLeft\int_M\left(|\nabla^2u|^2-(Lu)^2+(\Ric+\nabla^2U)(\nabla u,\nabla u)\right)\diff\mu\\
        &=\int_\dM\left((H+\partial_NU)\left(\partial_Nu\right)^2 - 2(\partial_Nu)\Ld u + h\big(\nablad u,\nablad u\big)\right)\diff\mudM\,,
    \end{align*}
    where $\Ld = \Deltad-\langle\nablad U,\nablad\rangle$. 
\end{restatable}
\Cref{lem:GenReilly} has been applied by Kolesnikov and Milman \cite{KolesnikovMilman2017Brascamp,KolesnikovMilman2016Riemannian,KolesnikovMilman2016Poincare} to obtain Brascamp-Lieb-type and log-Sobolev functional inequalities on weighted Riemannian manifolds with boundary.
\begin{rema}\label{rem:GenReillyCorners}
    The set of corner points $C$ of a Riemannian manifold $M$ with corners, see \Cref{ssec:ManifoldBdry}, forms a set of measure $0$ with respect to the Riemannian volume measure $\nu_M$, and $M\setminus C$ is a Riemannian manifold with boundary. Hence \Cref{lem:GenReilly} holds analogously for Riemannian manifolds with corners when replacing the domain of integration on the right-hand side with the boundary $\partial(M\setminus C)$ of $M\setminus C$.
\end{rema}

The Sobolev space $H^{1,2}(M,\mu)$ on $M$ is defined as the closure of $C^\infty(M)\cap L^2(\mu)$ in $L^2(\mu)$ with respect to the norm
\begin{equation}\label{eq:MSobolevDef1}
    \norm{u}_{H^{1,2}(\mu)}^2 = \norm{u}_{L^2(\mu)}^2 + \norm{\nabla u}_{L^2(\mu)}^2\,.
\end{equation}
Similarly, one defines the second-order Sobolev space $H^{2,2}(M,\mu)$ on $M$ as the closure of $C^\infty(M)\cap L^2(\mu)$ with respect to the norm
\begin{equation}\label{eq:MSobolevDef2}
    \norm{u}_{H^{2,2}(\mu)}^2 = \norm{u}_{L^2(\mu)}^2 + \norm{\nabla u}_{L^2(\mu)}^2 + \norm{\nabla^2 u}_{L^2(\mu)}^2\,,
\end{equation}
where $\norm{\nabla^2u}_{L^2(\mu)}$ is the $L^2(\mu)$-norm of $|\nabla^2u|$.

We make the following assumptions on $(M,g)$ and the probability measure $\mu$.
\begin{samepage}
\begin{assu}\label{ass:U}\mbox{}
    \begin{enumerate}[(i)]
        \item\label{ass:hdefinite}\textbf{Local convexity.} The scalar second fundamental form $h$ on $\dM$ is negative semi-definite, i.e.\ 
        \begin{equation*}
            h(v,v)\leq 0\qquad\textup{for all }x\in\dM\textup{ and }v\in \T_x\dM\,.
        \end{equation*}
        \item\label{ass:RicHess}\textbf{Lower curvature bound.} There is a constant $\rho\in [0,\infty )$ such that $\Ric+\nabla^2 U\geq-\rho$ on $M$, i.e.\ 
        \begin{equation*}
            \Ric(v,v)+\nabla^2U(v,v) \geq -\rho|v|^2\qquad\textup{for all }x\in M\textup{ and }v\in\T_xM\,.
        \end{equation*}
        \item\label{ass:poincare}\textbf{Poincaré inequality.} The probability measure $\mu$ satisfies a Poincaré inequality with constant $m^{-1}$, i.e.\  
        \begin{equation*}
            \int_{M}f^2\diff\mu\leq\frac{1}{m}\int_{M}|\nabla f|^2\diff\mu\, \quad\text{ for all $f\in H^{1,2}(M,\mu)$ with $\int_{M}f\diff\mu=0$}\,.
        \end{equation*}
    \end{enumerate}
\end{assu}
\end{samepage}

The local convexity assumption \eqref{ass:hdefinite} is implied by geodesic convexity of the manifold, i.e.\ the existence of a minimising geodesic between any two interior points that lies in the interior, while the converse is not true in general \cite{Bishop1974Convexity,Bishop1964Geometry,Alexander1977LocallyConvex}.
Note that the lower bound \eqref{ass:RicHess} on the generalised curvature tensor $\Ric+\nabla^2U$ is also known as the Bakry-Émery curvature-dimension condition CD($-\rho$,$\infty$) \cite{Bakry2014Analysis}. 

\begin{rema}\label{rem:Mconvex}
    Let $M$ be a closed, connected subset of $\R^d$ with smooth boundary. Then $M$ is a Riemannian manifold with boundary when equipped with the Euclidean metric and \Cref{ass:U}\eqref{ass:hdefinite} is equivalent to convexity of $M$. Since the Euclidean metric is flat, \Cref{ass:U}\eqref{ass:RicHess} simplifies to $\nabla^2U\geq-\rho$ on $M$. Hence \Cref{ass:U} is satisfied by probability measures on convex subsets of Euclidean space satisfying a Poincaré inequality and a lower curvature bound on their potential.
\end{rema}

By the integration-by-parts identity \eqref{eq:intbypartsL}, when equipped with the Neumann boundary conditions
\begin{equation}\label{eq:defD}
    \D=\{u\in C_\c^\infty(M)\colon\partial_Nu=0\textup{ on }\dM\}\,,
\end{equation}
we have $Lu = -\nabla^*\nabla u$ for all $u\in\D$, where the adjoint is in $L^2(\mu)$. Therefore, the densely defined linear operator $(L,\D)$ on $L^2(\mu)$ is symmetric and negative semi-definite.

\begin{coro}\label{cor:reillyestimate}
    Let \Cref{ass:U} hold.
    \begin{enumerate}[(i)]
        \item For all $u\in\D$,
            \begin{equation}\label{eq:HessianUpperBound}
                \norm{\nabla^2u}_{L^2(\mu)}^2 \leq \norm{Lu}_{L^2(\mu)}^2 + \rho\norm{\nabla u}_{L^2(\mu)}^2\,.
            \end{equation}
        \item The operator $(L,\D)$ is essentially self-adjoint. Its unique self-adjoint extension $(L,\dom(L))$ satisfies $\dom(L)\subseteq H^{2,2}(M,\mu)$ and \eqref{eq:HessianUpperBound} holds for all $u\in\dom(L)$.
    \end{enumerate}
    
\end{coro}
\begin{proof}\begin{enumerate}[(i)]
    \item The estimate \eqref{eq:HessianUpperBound} follows directly from the generalised Reilly formula stated in \Cref{lem:GenReilly}, using \Cref{ass:U}\eqref{ass:hdefinite} and \eqref{ass:RicHess} and the vanishing normal derivative at the boundary. 
    \item In the case of constant potential, the essential self-adjointness of the Laplace-Beltrami operator $(\Delta,\D)$ with Neumann boundary conditions is classical and e.g.\ shown in \cite[Prop. 8.2.5]{Taylor2023PDE}. Another proof following an approach by Chernoff \cite{Chernoff1973Essential} is given in \cite{BianchiGüneysuSetti2024Neumann}. Due to the integration-by-parts identity \eqref{eq:intbypartsL} and the regularity estimate \eqref{eq:HessianUpperBound} guaranteed by the lower bound on the curvature of $U$, the general case $L=\Delta-\langle\nabla U,\nabla\rangle$ can be treated analogously.
    
    Finally, since $\norm{\nabla u}_{L^2(\mu)}^2 \leq \norm{u}_{L^2(\mu)}\norm{Lu}_{L^2(\mu)}$ for all $u\in\D$, \eqref{eq:HessianUpperBound} yields 
    \begin{equation*}
        \norm{\nabla^2u}_{L^2(\mu)}^2\leq (1+\rho)(\norm{u}_{L^2(\mu)} + \norm{Lu}_{L^2(\mu)})^2
    \end{equation*}
    for all $u\in\D$. In particular, this implies $\dom(L)\subseteq H^{2,2}(M,\mu)$ and \eqref{eq:HessianUpperBound} holds for all $u\in\dom(L)$ by an approximation argument.
\end{enumerate}
\end{proof}

We make the following technical assumption which allows us to work with a spectral decomposition of $L$. It can surely be relaxed, see \cite{BrigatiStoltz2023Decay}, yet is satisfied in many cases, for instance if $M$ is compact.

\begin{assu}
    \label{ass:discretespec} The spectrum of $(L,\dom(L))$ on $L^2(\mu)$ is discrete.
\end{assu}
Let $L_0^2(\mu) = \{f\in L^2(\mu)\colon\int_M f\diff\mu=0\}$ denote the subspace of $L^2(\mu)$ of mean-zero functions. By the Poincaré inequality \Cref{ass:U}\eqref{ass:poincare}, 
\begin{equation*}
    -L|_{L_0^2(\mu)}\colon\dom(L)\cap L_0^2(\mu)\to L_0^2(\mu)
\end{equation*}
has a spectral gap and hence
admits a bounded linear inverse
\begin{equation}\label{eq:defG}
    G\colon L_0^2(\mu)\to \dom(L)\cap L_0^2(\mu)
\end{equation}
with operator norm bounded by $\frac{1}{m}$.
\Cref{ass:discretespec} allows us to consider an orthonormal basis $\{e_k\colon k\in\N_0\}$ of $L^2(\mu)$ consisting of eigenfunctions of $L$ with eigenvalues $-\alpha_k^2$, where $e_0=1$ and $\alpha_0=0$. In particular, $\inf\{\alpha_k^2\colon k\in\N\}\geq m$ and $e_k\in\dom(L)$ implies $\partial_Ne_k=0$ on $\dM$.
The operator $G=(-L|_{L_0^2(\mu)})^{-1}$ satisfies
\begin{equation*}
       Ge_k = \frac{1}{\alpha_k^2}e_k\qquad\textup{for all }k\in\N\,.
\end{equation*}

\section{The space-time divergence lemma}\label{sec:divergence}

We adjoin a time component to the spatial component and hence consider time-space domains $\Mq=[0,T]\times M$, where $T\in(0,\infty)$ is fixed and, as before, $(M,g)$ is a Riemannian manifold with boundary such that \Cref{ass:U} is satisfied. Note that $\Mq$ is a Riemannian manifold with corners with metric $\overline{g}$ given by the product of the Euclidean metric and $g$. Its corner points are $\{0,T\}\times\dM$. The tangent bundle $\T\Mq$ is canonically isomorphic to $([0,T]\times\R)\oplus \T M$, and we identify elements $V\in\T\Mq$ of the former with tuples $(V_0,V_1)$ with $V_0\in[0,T]\times\R$ and $V_1\in\T M$.

On $\Mq$ we consider the measure
\begin{equation*}
    \muq = \lambda\otimes\mu\,,
\end{equation*}
where $\lambda$ is the Lebesgue measure restricted to $[0,T]$. We use the notation $L^2(\muq) = L^2(\Mq,\muq)$ and $L_0^2(\muq) = \{f\in L^2(\muq)\colon\int_\Mq f\diff\muq=0\}$ for the subspace of $L^2(\muq)$ of mean-zero functions. The Riemannian gradient on $\overline{M}$ is denoted by $\nablaq$ and its formal $L^2(\muq)$-adjoint by $\nablaq^*$. The latter acts on a vector field $X\in\Gamma_{C^1}(\T\Mq)$ by
\begin{align*}
    \nablaq^*X &= -\partial_t X_0 + \nabla^*X_1\\
    &=-\partial_tX_0 - \divg(X_1) + \langle\nabla U,X_1\rangle\,.
\end{align*}
For $m\in\{1,2\}$, the Sobolev spaces $H^{m,2}(\muq)$ are defined in analogy to \eqref{eq:MSobolevDef1} and \eqref{eq:MSobolevDef2}.
Similarly, the Sobolev space $H^{m,2}_\DC(\muq)$ with Dirichlet boundary conditions in time is the closure of 
\begin{equation*}
    \{u\in C^\infty(\Mq)\cap L^2(\muq)\colon u(0,\cdot)=u(T,\cdot)=0\}
\end{equation*}
in $L^2(\muq)$ with respect to $\norm{\cdot}_{H^{m,2}(\muq)}$.

\begin{theo}[Quantitative divergence lemma]\label{thm:divergence}
    Suppose that \Cref{ass:U,ass:discretespec} are satisfied.
    Then for any $T\in (0,\infty )$, there exist constants $c_0(T),c_1(T)\in (0,\infty )$ such that for any $f\in L_0^2(\muq)$ there are functions
    \begin{equation}\label{eq:hgDirichletcond}
        h\in H^{1,2}_\DC(\muq)\qquad\textup{and}\qquad g\in H^{2,2}_\DC(\muq)
    \end{equation}
    satisfying 
    \begin{equation}\label{eq:gdomLcond}
        g(t,\cdot)\in\dom(L)\qquad\textup{for }\lambda\textup{-almost all }t\in[0,T]
    \end{equation}
    such that
    \begin{equation}
        f\ =\ \partial_th-Lg
    \end{equation}
    and the functions $h$ and $g$ satisfy
    \begin{eqnarray}\label{eq:bound0order}
        \norm{h}_{L^2(\muq)}^2+\norm{\nabla g}_{L^2(\muq)}^2 &\leq& c_0(T)\,\norm{f}_{L^2(\muq)}^2\, ,\\\label{eq:bound1order}
        \norm{\nablaq h}_{L^2(\muq)}^2+\norm{\nablaq\nabla g}_{L^2(\muq)}^2 &\leq& c_1(T)\,\norm{f}_{L^2(\muq)}^2\,,
    \end{eqnarray} 
    where
    \begin{equation}\label{eq:c0c1}
        \begin{aligned}
            c_0(T)\, &=\, 2T^2+43\frac{1}{m}\qquad\textup{and}\\
            c_1(T)\, &=\, 290+\frac{991}{mT^2}+43\max\left(\frac{1}{m},\frac{T^2}{\pi^2}\right)\rho\,.
        \end{aligned}
    \end{equation}

\end{theo}

\begin{rema}[Relation to the classical divergence lemma]
    In \Cref{thm:divergence}, writing 
    \begin{equation*}
        X=\begin{pmatrix}X_0\\ X_1\end{pmatrix}\qquad\textup{with }X_0=-h\textup{ and }X_1=\nabla g
    \end{equation*}
    yields the divergence equation
    \begin{equation}\label{eq:divergence}
        f \ =\  \nablaq^*X
    \end{equation}
    together with the bounds
    \begin{eqnarray}\label{eq:boundX}
        \norm{X}_{L^2(\muq)}^2 &\leq& c_0(T)\,\norm{f}_{L^2(\muq)}^2\, ,\\
        \label{eq:boundnablaX}
        \norm{\overline\nabla X}_{L^2(\muq)}^2 &\leq& c_1(T)\,\norm{f}_{L^2(\muq)}^2\,.
    \end{eqnarray} 
    The vector field $X$ satisfies Dirichlet boundary conditions on the time boundary, i.e.\ it vanishes on $\{0,T\}\times M$, and it vanishes in the direction normal to the spatial boundary, i.e.\ $\langle X_1,N\rangle=0$ on $[0,T]\times\dM$.
    In comparison, in the classical Lions' divergence lemma as in \cite{Amrouche2015Lions}, one imposes Dirichlet boundary conditions on the solution to \eqref{eq:divergence} on the whole boundary. While \Cref{thm:divergence} only requires the vector field to vanish in the direction normal to the spatial boundary, we obtain the additional gradient structure of $X_1$.
\end{rema}
\begin{rema}[Related works]
    In the unbounded Euclidean case $M=\R^d$, a quantitative divergence lemma has been proved by Cao, Lu and Wang \cite[Lemma 2.6]{Cao2019Langevin}. We obtain the order of the constant $c_1(T)$ as conjectured in \cite[Remark 2.7]{Cao2019Langevin} by more closely considering the case of space-time harmonic right-hand sides and splitting these into high and low modes, as well as into symmetric and antisymmetric functions. Again on $\R^d$, the recent preprint \cite{Lehec2024Convergence} takes a variational approach to the divergence lemma, yet obtains an additional term of the order $mT^2$ in the constant $c_1(T)$ in \eqref{eq:c0c1} due to the estimate of $\norm{\nabla h}_{L^2(\muq)}^2$. We avoid this, at the cost of the estimate of $\norm{\partial_t\nabla g}_{L^2(\muq)}^2$ being of the order $1+\frac{1}{mT^2}$ instead of $\frac{1}{mT^2}$ as in \cite{Lehec2024Convergence}. Brigati and Stoltz \cite{Brigati2022FokkerPlanck,BrigatiStoltz2023Decay} prove a related averaging lemma, eliminating the need of discrete spectrum of $L$, at the cost of worse scaling in $T$. Finally, \cite{EGHLM2024Lifting} uses a variation of our proof by formulating the divergence lemma in terms of the associated Dirichlet form, allowing for the treatment of lifts of diffusions with different boundary behaviour.
\end{rema}

Let us first give a sketch of the proof of \Cref{thm:divergence}. A first idea would be to solve the equation
\begin{equation}\label{eq:Lbaruf}
    (\partial_t^2+L)u = f
\end{equation}
with Neumann boundary conditions and to set $X=-\overline{\nabla}u$, or, equivalently, $h=\partial_tu$ and $g = -u$. However, by definition, Neumann boundary conditions only lead to $h$ and $\partial_tg$ vanishing on the time boundary and not $g$ as required by \eqref{eq:hgDirichletcond}. More precisely, let
\begin{align*}
    \overline\D&=\left\{u\in C^\infty(\Mq)\cap L^2(\muq)\colon\partial_Nu=0\textup{ on }[0,T]\times\dM,\,\partial_tu=0\textup{ on }\{0,T\}\times M\right\}\,,\\
    \overline{L}u &= (\partial_t^2+L)u = -{\nablaq^*\nablaq} u\qquad\textup{for }u\in\overline{\D}\,.
\end{align*}
Then, as in \Cref{cor:reillyestimate}, the operator $(\overline{L},\overline{\D})$ on $L^2(\muq)$ is essentially self-adjoint and its unique self-adjoint extension $(\overline{L},\dom(\overline{L}))$ satisfies $\dom(\overline{L})\subseteq H^{2,2}(\muq)$. In particular, $\dom(\overline{L})$ contains functions $u\in L^2(\muq)$ such that $u(t,\cdot)\in\dom(L)$ for $\lambda$-a.e.\ $t\in[0,T]$ and $u(\cdot,x)\in H^{2,2}([0,T])$ for $\muq$-a.e.\ $x\in M$ satisfying the Neumann boundary conditions $\partial_tu(0,\cdot) = \partial_tu(T,\cdot) = 0$ and $(\partial_t^2+L)u\in L^2(\muq)$. 

The operator $(\overline{L},\dom(\overline{L}))$ has discrete spectrum on $L^2(\muq)$ and satisfies a Poincaré inequality with constant $\overline{m}= \min(m,\frac{\pi^2}{T^2})$. 
Therefore, the restriction
\begin{equation*}
    -\overline{L}|_{L_0^2(\muq)}\colon\dom(\overline{L})\cap L_0^2(\muq)\to L_0^2(\muq)
\end{equation*}
admits a bounded linear inverse
\begin{equation*}
    \overline{G}\colon L_0^2(\muq) \to \dom(\overline{L})\cap L_0^2(\muq)
\end{equation*}
whose operator norm is bounded from above by $\overline{m}^{-1} = \max(\frac{1}{m},\frac{T^2}{\pi^2})$.

Hence $u=\overline{G}f$ is a solution of \eqref{eq:Lbaruf} and bounds as in \eqref{eq:bound0order} and \eqref{eq:bound1order} can be derived for $h=-\partial_tu$ and $g=u$ in a straightforward way. However, in general $u$ does not satisfy Dirichlet boundary conditions in time as required in \eqref{eq:hgDirichletcond}. Indeed, an integration by parts shows that Dirichlet boundary conditions in time hold for $u$ if and only if $f$ is orthogonal to the space $\H$ of space-time harmonic functions, that is $\H$ consists of all functions $u\in L^2(\muq)$ such that $u(t,\cdot)\in\dom(L)$ for $\lambda$-a.e.\ $t\in[0,T]$ and $u(\cdot,x)\in H^{2,2}([0,T])$ for $\muq$-a.e.\ $x\in M$ and $(\partial_t^2+L)u=0$. 

Therefore, in the proof of \Cref{thm:divergence}, we treat right-hand sides in $\H_0 = \H\cap L_0^2(\muq)$ and its orthogonal complement $\H_0^\bot$ in $L_0^2(\muq)$ separately. 
For functions $f\in\H_0^\bot$, we can set $u = \overline{G}f\in H^{2,2}(\muq)$ and $X = \nablaq u$, or equivalently $h=-\partial_tu$ and $g=u$, to obtain $\nablaq^* X = \partial_th-Lg = f$. The orthogonality of $f$ to the space of harmonic functions $\H_0$ then yields the boundary conditions \eqref{eq:hgDirichletcond}, and the bounds \eqref{eq:bound0order} and \eqref{eq:bound1order} follow by the Poincaré inequality and the Reilly formula.
Indeed, by \Cref{rem:GenReillyCorners}, we can apply \Cref{lem:GenReilly} on the Riemannian manifold $\Mq$ with corners, so that we obtain the estimate
\begin{equation}\label{eq:HessianBarUpperBound}
    \norm{\nablaq^2u}_{L^2(\muq)}^2 \leq \norm{\overline{L}u}_{L^2(\muq)}^2 + \rho\norm{\nablaq u}_{L^2(\muq)}^2
\end{equation}
for any function $u\in \dom(\overline{L})$ in analogy to \Cref{cor:reillyestimate}.

In order to treat functions $f\in\H_0$, we decompose $\H_0$ using an orthogonal basis of $L^2(\mu)$ consisting of eigenfunctions of $L$. We then consider the antisymmetric and symmetric parts as well as the high and low modes separately. In the case of antisymmetric functions in the low modes, $h$ and $g$ can simply be chosen as the time-integral of $f$ and zero. The tricky case is that of symmetric functions in the low modes: this is where the estimates degenerate for $T\to 0$. Finally, the high modes are not expected to effect the results. This can indeed be verified by a technical computation.

\begin{proof}[Proof of \Cref{thm:divergence}]
    Let $\H_0 = \H\cap L_0^2(\mu)$ as defined above and decompose 
    \begin{equation*}
        L_0^2(\muq) = \H_0\oplus\H_0^\bot\,.
    \end{equation*}
    At first consider $f\in \H_0^\bot$. Then setting $h=-\partial_tu$ and $g=u$ with $u = \overline{G}f$ yields $u\in\dom(\overline{L})$ and 
    \begin{equation*}
        f = \partial_th-Lg\,.
    \end{equation*}
    Furthermore, for any $v\in\H_0$, an integration by parts yields
    \begin{align*}
        0 = (f,v)_{L^2(\muq)} &= (\overline{L}u,v)_{L^2(\muq)}\\
        &= (u,(\partial_t^2+L)v)_{L^2(\muq)} + \int_{(0,T)}\int_{\dM}v\partial_Nu-u\partial_{N}v\diff\mudM\diff\lambda\\
        &\qquad + \int_M\big(u(T,\cdot)\partial_tv(T,\cdot) - u(0,\cdot)\partial_tv(0,\cdot)\big)\diff\mu\\
        &=\int_M\big(u(T,\cdot)\partial_tv(T,\cdot) - u(0,\cdot)\partial_tv(0,\cdot)\big)\diff\mu\,.
    \end{align*}
    Hence $u = 0$ on $\{0,T\}\times M$ since $v\in\H_0$ is arbitrary. Since by definition $u$ also satisfies Neumann boundary conditions, $h$ and $g$ and the boundary conditions \eqref{eq:hgDirichletcond} and \eqref{eq:gdomLcond}.
    We obtain the estimates
    \begin{equation*}
        \norm{h}_{L^2(\muq)}^2 + \norm{\nabla g}_{L^2(\muq)}^2 = \norm{\nablaq u}_{L^2(\muq)}^2  = (u,f)_{L^2(\muq)}\leq \max\left(\frac{1}{m},\frac{T^2}{\pi^2}\right)\norm{f}_{L^2(\muq)}^2
    \end{equation*}
    and
    \begin{align*}
        \MoveEqLeft\norm{\partial_th}_{L^2(\muq)}^2 + \norm{\partial_t\nabla g}_{L^2(\muq)}^2 + \norm{\nabla h}_{L^2(\muq)}^2 + \norm{\nabla^2 g}_{L^2(\muq)}^2\\
        &= \norm{\overline{\nabla}^2u}_{L^2(\muq)}^2\leq \norm{\overline{L}u}_{L^2(\muq)}^2 + \rho\norm{\overline{\nabla}u}_{L^2(\muq)}^2\\
        &\leq \left(1+\max\left(\frac{1}{m},\frac{T^2}{\pi^2}\right)\rho\right)\norm{f}_{L^2(\muq)}^2
    \end{align*}
    by \eqref{eq:HessianBarUpperBound} and $\norm{u}_{L^2(\muq)}\leq\max(\frac{1}{m},\frac{T^2}{\pi^2})\norm{f}_{L^2(\muq)}$. Therefore, for functions $f\in \H_0^\bot$, the bounds \eqref{eq:bound0order} and \eqref{eq:bound1order} are satisfied with the constants
    \begin{equation*}
        c_0^\bot = \max\left(\frac{1}{m},\frac{T^2}{\pi^2}\right)\qquad\textup{and}\qquad c_1^\bot = 1+\max\left(\frac{1}{m},\frac{T^2}{\pi^2}\right)\rho\,.
    \end{equation*}

    Now consider $f\in\H_0$. In terms of the orthonormal basis $(e_k)_{k\in\N_0}$ of eigenfunctions of $L$, any function $f\in\H_0$ can be represented as
    \begin{equation*}
        f(t,x) = \sum_{k\in\N_0}f_k(t)e_k(x)\,,
    \end{equation*}
    where $f_k = (f,e_k)_{L^2(\mu)}$. This yields
    \begin{equation*}
        0 = (\partial_t^2f+Lf)(t,x) = f_0''(t) + \sum_{k\in\N}(f_k''(t)-\alpha_k^2f_k(t))e_k(x)\,,
    \end{equation*}
    so that the coefficients $f_k$ satisfy the ordinary differential equations $f_k'' = \alpha_k^2 f_k$. Hence for any $k\in\N$, the function $f_k$ can be expressed as a linear combination of $e^{-\alpha_kt}$ and $e^{-\alpha_k(T-t)}$, and the functions
    \begin{align*}
        H_0^a(t,x)&=(t-(T-t))e_0(x)\,,\\
        H_k^a(t,x)&=(e^{-\alpha_kt}-e^{-\alpha_k(T-t)})e_k(x)\quad\textup{for }k\in\N\,,\\
        H_k^s(t,x)&=(e^{-\alpha_kt}+e^{-\alpha_k(T-t)})e_k(x)\quad\textup{for }k\in\N
    \end{align*}
    define an orthogonal basis $\{H_k^a\colon k\in\N_0\}\cup\{H_k^s\colon k\in\N\}$ of $\H_0$. The functions $H_k^a$ and $H_k^s$ are in $H^{2,2}(\muq)$ since $e_k\in\dom(L)\subseteq H^{2,2}(\mu)$ for all $k\in\N_0$. We decompose $\H_0$ into the symmetric and antisymmetric functions as well as into the high and low modes, i.e.\ 
    \begin{equation*}
        \H_0=\H_{l,a}\oplus\H_{l,s}\oplus\H_{h,a}\oplus\H_{h,s}\,,
    \end{equation*}
    where
    \begin{align*}
        \H_{l,a}&=\spn\{H_k^a\colon k\in\N_0\textup{ with }\alpha_k\leq\frac{\beta}{T}\}\,,\\
        \H_{l,s}&=\spn\{H_k^s\colon k\in\N\textup{ with }\alpha_k\leq\frac{\beta}{T}\}\,,\\
        \H_{h,a}&=\overline{\spn}\{H_k^a\colon k\in\N\textup{ with }\alpha_k>\frac{\beta}{T}\}\,,\\
        \H_{h,s}&=\overline{\spn}\{H_k^s\colon k\in\N\textup{ with }\alpha_k>\frac{\beta}{T}\}\,,
    \end{align*}
    and consider these four subspaces separately.
    Here $\beta\geq 2$ is the cutoff between high and low modes. For $f\in\H_{l,a}\mathbin{\cup}\H_{l,s}$ we have $\left(f,-Lf\right)_{L^2(\muq)}\leq\frac{\beta^2}{T^2}\norm{f}^2_{L^2(\muq)}$.    
    We will later choose $\beta=2$.

    \emph{Case 1:} Let $f\in \H_{l,a}$. Set $h(t,x)=\int_0^tf(s,x)\diff s$ and $g(t,x)=0$. Then $h\in H_\DC^{1,2}(\muq)$ due to the antisymmetry of $f$, so that \eqref{eq:hgDirichletcond} and \eqref{eq:gdomLcond} are satisfied. Since $\partial_th=f$, the Poincaré inequality on $[0,T]$ yields
    \begin{align*}
        \norm{h}_{L^2(\muq)}^2 &\leq \frac{T^2}{\pi^2}\norm{\partial_th}_{L^2(\muq)}^2 = \frac{T^2}{\pi^2}\norm{f}_{L^2(\mu)}^2\,,\\
        \norm{\nabla h}_{L^2(\muq)}^2 &\leq \frac{T^2}{\pi^2}\norm{\nabla f}_{L^2(\muq)}^2 = \frac{T^2}{\pi^2}\int_0^T\big(f(t,\cdot),-Lf(t,\cdot)\big)_{L^2(\mu)}\diff t\\
        &\leq\frac{T^2}{\pi^2}\int_0^T\frac{\beta^2}{T^2}\norm{f(t,\cdot)}_{L^2(\mu)}^2\diff t = \frac{\beta^2}{\pi^2}\norm{f}_{L^2(\muq)}\,.
    \end{align*}
    Hence for $f\in\H_{l,a}$ we obtain \eqref{eq:bound0order} and \eqref{eq:bound1order} with the constants
    \begin{align*}
        c_0^{l,a} = \frac{T^2}{\pi^2}\qquad\textup{and}\qquad c_1^{l,a} = 1 + \frac{\beta^2}{\pi^2}\,.\\
    \end{align*}
    
    \emph{Case 2:} Let $f\in\H_{l,s}$. Since $f$ is symmetric, it does not necessarily integrate to $0$ in time and we cannot argue as in the previous case. Instead, we consider the decomposition $f = f_0+f_1$,
    \begin{align*}
        f_0(t,x)&=f(0,x)\cdot\cos\left(\frac{2\pi t}{T}\right)\,,\\
        f_1(t,x)&=f(t,x)-f_0(t,x)
    \end{align*}
    of $f$ into a part $f_0$ that integrates to $0$ in time and a part $f_1$ with Dirichlet boundary values in time.
    Note that $f(t,\cdot),f_0(t,\cdot),f_1(t,\cdot)$ are contained in $\spn\{e_k\colon k\in\N\textup{ with }\alpha_k\leq\frac{\beta}{T}\}$ for all $t\in[0,T]$. We now set 
    \begin{equation*}
        h(t,x)=\int_0^tf_0(s,x)\diff s\quad\textup{and}\quad g(t,x)=Gf_1(t,\cdot)(x)\,,
    \end{equation*}
    so that $f_0=\partial_th$ and $f_1=-Lg$. The boundary conditions \eqref{eq:hgDirichletcond} are satisfied due to the boundary conditions of $f_0$ and $f_1$, and \eqref{eq:gdomLcond} is satisfied by the definition \eqref{eq:defG} of $G$.
    In order to obtain the required estimates, we bound $\norm{f_0}_{L^2(\muq)}$ by $\norm{f}_{L^2(\muq)}$. First note that
    \begin{equation*}
        \norm{f_0}_{L^2(\muq)}^2 = \norm{f(0,\cdot)}_{L^2(\mu)}^2\int_0^T\cos\left(\frac{2\pi t}{T}\right)^2\diff t = \frac{T}{2}\norm{f(0,\cdot)}_{L^2(\mu)}^2\,.
    \end{equation*}
    Expanding $f(t,x) = \sum_{\alpha_k\leq\frac{\beta}{T}}b_kH_k^s(t,x)$ yields
    \begin{align*}
        \norm{f(0,\cdot)}_{L^2(\mu)}^2 = \sum_{\alpha_k\leq\frac{\beta}{T}}b_k^2(1+e^{-\alpha_kT})^2\leq 4\sum_{\alpha_k\leq\frac{\beta}{T}}b_k^2
    \end{align*}
    and
    \begin{align*}
        \norm{f}_{L^2(\muq)}^2 &= \sum_{\alpha_k\leq\frac{\beta}{T}}b_k^2\int_0^T\left(e^{-\alpha_kt}+e^{-\alpha_k(T-t)}\right)^2\diff t\\
        &\geq \sum_{\alpha_k\leq\frac{\beta}{T}}T\min_{t\in[0,T]}\left(e^{-\alpha_kt}+e^{-\alpha_k(T-t)}\right)^2b_k^2\\
        &=\sum_{\alpha_k\leq\frac{\beta}{T}}4Te^{-\alpha_kT}b_k^2 \geq 4Te^{-\beta}\sum_{\alpha_k\leq\frac{\beta}{T}}b_k^2\,,
    \end{align*}
    so that
    \begin{equation}\label{eq:Hls:boundf0}
        \norm{f_0}_{L^2(\muq)}^2\leq\frac{T}{2}\norm{f(0,\cdot)}_{L^2(\mu)}^2\leq\frac{1}{2}e^{\beta}\norm{f}_{L^2(\muq)}^2
    \end{equation}
    and
    \begin{equation*}
        \norm{f_1}_{L^2(\muq)}^2 \leq \left(\norm{f}_{L^2(\muq)}+\norm{f_0}_{L^2(\muq)}\right)^2\leq \big(1+\frac{1}{\sqrt{2}}e^{\beta/2}\big)^2\norm{f}_{L^2(\muq)}^2\leq (2+e^\beta)\norm{f}_{L^2(\muq)}^2\,.
    \end{equation*}
    Similarly to the previous case, this gives the estimates
    \begin{align*}
        \norm{h}_{L^2(\muq)}^2 &\leq \frac{T^2}{\pi^2}\norm{f_0}_{L^2(\muq)}^2\leq\frac{T^2}{2\pi^2}e^\beta\norm{f}_{L^2(\muq)}^2\,,\\
        \norm{\partial_th}_{L^2(\muq)}^2 &= \norm{f_0}_{L^2(\muq)}^2 \leq \frac{1}{2}e^{\beta}\norm{f}_{L^2(\muq)}^2\,,\\
        \norm{\nabla h}_{L^2(\muq)}^2 &\leq\frac{\beta^2}{\pi^2}\norm{f_0}_{L^2(\muq)}^2 \leq \frac{\beta^2}{2\pi^2}e^\beta\norm{f}_{L^2(\muq)}^2\,.
    \end{align*}
    For the bounds on $g$ we use the fact that 
    \begin{equation*}
        \norm{\nabla Gu}_{L^2(\mu)}^2 = (u,Gu)_{L^2(\mu)}\leq\frac{1}{m}\norm{u}_{L^2(\mu)}^2
    \end{equation*}
    for any $u\in L_0^2(\mu)$. Then
    \begin{align*}
        \norm{\nabla g}_{L^2(\muq)}^2 = \int_0^T\norm{\nabla Gf_1(t,\cdot)}_{L^2(\mu)}^2\diff t\leq\frac{1}{m}\norm{f_1}_{L^2(\muq)}^2\leq \frac{1}{m}\big(1+\frac{1}{\sqrt{2}}e^{\beta/2}\big)^2\norm{f}_{L^2(\muq)}^2\,.
    \end{align*}
    Similarly, commuting the derivatives,
    \begin{align}\label{eq:Hls:dtf1}
        \norm{\partial_t\nabla g}_{L^2(\muq)}^2 = \norm{\nabla\partial_tg}_{L^2(\muq)}^2\leq\frac{1}{m}\norm{\partial_tf_1}_{L^2(\muq)}^2\leq\frac{1}{m}\left(\norm{\partial_tf}_{L^2(\muq)} + \norm{\partial_tf_0}_{L^2(\muq)}\right)^2\,.
    \end{align}
    Again using the expansion $f(t,x) = \sum_{\alpha_k\leq\frac{\beta}{T}}b_kh_k^s(t,x)$ we obtain
    \begin{align*}
        \norm{\partial_tf}_{L^2(\muq)}^2 &= \sum_{\alpha_k\leq\frac{\beta}{T}}\alpha_k^2b_k^2\int_0^T\left(-e^{-\alpha_kt}+e^{-\alpha_k(T-t)}\right)^2\diff t\\
        &\leq \frac{\beta^2}{T^2}\sum_{\alpha_k\leq\frac{\beta}{T}}b_k^2\int_0^T\left(e^{-\alpha_kt}+e^{-\alpha_k(T-t)}\right)^2\diff t = \frac{\beta^2}{T^2}\norm{f}_{L^2(\muq)}^2
    \end{align*}
    as well as
    \begin{equation*}
        \norm{\partial_tf_0}_{L^2(\muq)}^2 = \frac{2\pi^2}{T}\norm{f(0,\cdot)}_{L^2(\mu)}^2 \leq \frac{2\pi^2}{T^2}e^\beta\norm{f}_{L^2(\muq)}^2
    \end{equation*}
    by \eqref{eq:Hls:boundf0}. Hence \eqref{eq:Hls:dtf1} yields
    \begin{equation*}
        \norm{\partial_t\nabla g}_{L^2(\muq)}^2 \leq \frac{(\beta+\sqrt{2}\pi e^{\beta/2})^2}{mT^2}\norm{f}_{L^2(\muq)}^2.
    \end{equation*}
    Finally, \Cref{cor:reillyestimate} yields
    \begin{align*}
        \norm{\nabla^2g}_{L^2(\muq)}^2 &\leq \norm{Lg}_{L^2(\muq)}^2+\rho\norm{\nabla g}_{L^2(\muq)}^2\leq \left(1+\frac{\rho}{m}\right)\norm{f_1}_{L^2(\muq)}^2\\
        &\leq\left(1+\frac{\rho}{m}\right)\left(1+\frac{1}{\sqrt{2}}e^{\beta/2}\right)^2\norm{f}_{L^2(\muq)}\,.
    \end{align*}
    For $f\in\H_{l,s}$ we therefore obtain \eqref{eq:bound0order} and \eqref{eq:bound1order} with the constants
    \begin{align*}
        c_0^{l,s} &= \frac{T^2}{2\pi^2}e^\beta + \frac{1}{m}\left(1+\frac{1}{\sqrt{2}}e^{\beta/2}\right)^2\qquad\textup{and}\\
        c_1^{l,s} &= \left(\frac{1}{2}+\frac{\beta^2}{2\pi^2}\right)e^\beta + \frac{(\beta+\sqrt{2}\pi e^{\beta/2})^2}{mT^2} + \left(1+\frac{\rho}{m}\right)\left(1+\frac{1}{\sqrt{2}}e^{\beta/2}\right)^2\,.\\
    \end{align*}

    \emph{Case 3:} Let $f\in\H_{h,a}$. We will use the expansion $f(t,x) = \sum_{\alpha_k>\frac{\beta}{T}}b_kH_k^a(t,x)$. As we will show below, it is sufficient to derive a representation
    \begin{equation*}
        H_k^a = \partial_th_k-Lg_k
    \end{equation*}
    for each of the basis functions $H_k^a$ with $k\in\N$ fixed. To this end, we write  $u_k(t) = e^{-\alpha_kt}-e^{-\alpha_k(T-t)}$, so that $H_k^a(t,x) = u_k(t)e_k(x)$. We employ the ansatz 
    \begin{equation}\label{eq:ansatzuvw_a}
        u_k = \dot v_k+w_k\quad\textup{with }v_k(0)=v_k(T)=w_k(0)=w_k(T)=0\,.
    \end{equation}
    Then setting 
    \begin{equation*}
        h_k(t,x) = v_k(t)e_k(x)\qquad\textup{and}\qquad g_k(t,x) = \frac{1}{\alpha_k^2}w_k(t)e_k(x)
    \end{equation*}
    yields $H_k^a = \partial_th_k-Lg_k$. 
    
    We now construct such functions $v_k$ and $w_k$ for a fixed $k\in\N$. Let 
    \begin{equation*}
        \phi_k(t)=(\alpha_kt-1)^21_{[0,\alpha_k^{-1}]}(t)\,.
    \end{equation*}
    This function is in $C^1([0,T])$ and piecewise $C^2([0,T])$. Moreover, since $\alpha_k>\frac{2}{T}$, it satisfies $\phi_k(\frac{T}{2})=0$. For $t\in[0,\frac{T}{2}]$ we set
    \begin{align*}
        v_k(t)&=\phi_k(t)\int_0^tu_k(s)\diff s\qquad\textup{and}\\
        w_k(t)&=u_k(t)-\dot v_k(t)=(1-\phi_k(t))u_k(t)-\dot\phi_k(t)\int_0^tu_k(s)\diff s\,.
    \end{align*}
    Analogously, for $t\in[\frac{T}{2},T]$ we set
    \begin{align*}
        v_k(t)&=-\phi_k(T-t)\int_t^Tu_k(s)\diff s\qquad\textup{and}\\
        w_k(t)&=u_k(t)-\dot v_k(t)=(1-\phi_k(T-t))u_k(t)-\dot\phi_k(T-t)\int_t^Tu_k(s)\diff s\,.
    \end{align*}
    Then $v_k,w_k$ are in $C^0([0,T])$ and piecewise $C^1([0,T])$ with $v_k(0)=v_k(T)=0$,  $v_k(\frac{T}{2})=\dot v_k(\frac{T}{2})=0$, and
    \begin{align*}
        \begin{split}
            \dot v_k(0)&=\phi_k(0)u_k(0)=u_k(0)\,,\\
            w_k(0)&=u_k(0)-\dot v_k(0)=0\,,
        \end{split}
        \begin{split}
            \dot v_k(T)&=\phi_k(0)u_k(T)=u_k(T)\,,\\
        w_k(T)&=u_k(t)-\dot v_k(T)=0\,,
        \end{split}
    \end{align*}
    so that \eqref{eq:ansatzuvw_a} is satisfied. To obtain the required bounds, notice that $\norm{H_k^a}_{L^2(\muq)}^2 = \int_0^Tu_k(t)^2\diff t$ and
    \begin{equation*}
        \begin{split}
            \norm{h_k}_{L^2(\muq)}^2 &= \int_0^Tv_k(t)^2\diff t\,,\\
            \norm{\partial_th_k}_{L^2(\muq)}^2 &= \int_0^T\dot v_k(t)^2\diff t\,,\\
            \norm{\nabla h_k}_{L^2(\muq)}^2 &= \alpha_k^2\int_0^Tv_k(t)^2\diff t\,,
        \end{split}
        \qquad
        \begin{split}
            \norm{\nabla g_k}_{L^2(\muq)}^2 &= \frac{1}{\alpha_k^2}\int_0^Tw_k(t)^2\diff t\,,\\
            \norm{\partial_t\nabla g_k}_{L^2(\muq)}^2 &= \frac{1}{\alpha_k^2}\int_0^T\dot w_k(t)^2\diff t\,,\\
            \norm{Lg_k}_{L^2(\muq)}^2 &= \int_0^Tw_k(t)^2\diff t\,,
        \end{split}
    \end{equation*}
    where we used that $\norm{\nabla g_k}_{L^2(\muq)}^2=(g_k,-Lg_k)_{L^2(\muq)}$ and $\norm{\partial_t\nabla g_k}_{L^2(\muq)}^2 = (\partial_tg_k,-L\partial_tg_k)_{L^2(\muq)}$. One estimates
    \begin{align*}
        \int_0^{\frac{T}{2}}v_k(t)^2\diff t &= \int_0^\frac{1}{\alpha_k}\phi_k(t)^2\left(\int_0^tu_k(s)\diff s\right)^2\diff t \leq \int_0^\frac{1}{\alpha_k}(\alpha_kt-1)^4t\diff t\int_0^\frac{T}{2}u_k(t)^2\diff t\\
        &=\frac{1}{30\alpha_k^2}\int_0^\frac{T}{2}u_k(t)^2\diff t\,,
    \end{align*}
    and analogously on $[\frac{T}{2},T]$, so that 
    \begin{equation*}
        \int_0^Tv_k(t)^2\diff t\leq\frac{1}{30\alpha_k^2}\int_0^Tu_k(t)^2\diff t\,.
    \end{equation*}
    For $\dot v_k = \dot\phi_k\int_0^\cdot u_k(s)\diff s+\phi_ku_k$ we similarly estimate
    \begin{align*}
        \int_0^\frac{T}{2}\dot v_k(t)^2\diff t &\leq \left(\left(\int_0^\frac{1}{\alpha_k}\dot\phi_k(t)^2\left(\int_0^tu_k(s)\diff s\right)^2\diff t\right)^\frac{1}{2} + \left(\int_0^\frac{T}{2}u_k(t)^2\diff t\right)^\frac{1}{2}\right)^2\\
        &\leq \left(1+\left(\int_0^\frac{1}{\alpha_k}(2\alpha_k(\alpha_kt-1))^2t\diff t\right)^\frac{1}{2}\right)^2\int_0^\frac{T}{2}u_k(t)^2\diff t\\
        &=\left(1+\frac{1}{\sqrt{3}}\right)^2\int_0^\frac{T}{2}u_k(t)^2\diff t < \frac{5}{2}\int_0^\frac{T}{2}u_k(t)^2\diff t\,,
    \end{align*}
    so that 
    \begin{equation*}
        \int_0^T\dot v_k(t)^2\diff t \leq \left(1+\frac{1}{\sqrt{3}}\right)^2\int_0^Tu_k(t)^2\diff t\,.
    \end{equation*}
    Similarly, since $w_k = -\dot\phi_k\int_0^\cdot u_k(s)\diff s+(1-\phi_k)u_k$ on $[0,\frac{T}{2}]$ and correspondingly on $[\frac{T}{2},T]$, one obtains
    \begin{equation*}
        \int_0^Tw_k(t)^2\diff t\leq \left(1+\frac{1}{\sqrt{3}}\right)^2\int_0^Tu_k(t)^2\diff t\,.
    \end{equation*}
    For $\dot w_k = (1-\phi_k)\dot u_k -2\dot\phi_ku_k-\ddot\phi_k\int_0^\cdot u_k(s)\diff s$ we first estimate
    $\int_0^T\dot u_k(t)^2\diff t$ by $\int_0^Tu_k(t)^2\diff t$. A direct computation yields
    \begin{align*}
        \int_0^\frac{T}{2}\dot u_k(t)^2\diff t &= \alpha_k^2\int_0^\frac{T}{2}\left(e^{-\alpha_kt}+e^{-\alpha_k(T-t)}\right)^2\diff t\\
        &=\alpha_ke^{-\alpha_kT}(\sinh(\alpha_kT)+\alpha_kT)
    \end{align*}
    and similarly
    \begin{equation*}
        \int_0^\frac{T}{2}u_k(t)^2\diff t = \alpha_k^{-1}e^{-\alpha_kT}(\sinh(\alpha_kT)-\alpha_kT)\,,
    \end{equation*}
    so that
    \begin{equation*}
        \int_0^\frac{T}{2}\dot u_k(t)^2\diff t \leq \alpha_k^2\frac{\sinh(\beta)+\beta}{\sinh(\beta)-\beta} \int_0^\frac{T}{2}u_k(t)^2\diff t
    \end{equation*}
    by monotonicity of $x\mapsto\frac{\sinh(x)+x}{\sinh(x)-x}$ on $[0,\infty)$. Therefore, by similar arguments as before,
    \begin{align*}
        \int_0^\frac{T}{2}\dot w_k(t)^2\diff t
        &\leq \left(4\alpha_k+\alpha_k\left(\frac{\sinh(\beta)+\beta}{\sinh(\beta)-\beta}\right)^\frac{1}{2}+\left(\int_0^\frac{1}{\alpha_k}4\alpha_k^4t\diff t\right)^\frac{1}{2}\right)^2\int_0^\frac{T}{2}u_k(t)^2\diff t\\
        &=\alpha_k^2\left(4+\sqrt{2}+\left(\frac{\sinh(\beta)+\beta}{\sinh(\beta)-\beta}\right)^\frac{1}{2}\right)^2\int_0^\frac{T}{2}u_k(t)^2\diff t\,,
    \end{align*}
    and analogously on $[\frac{T}{2},T]$, yielding
    \begin{equation*}
        \int_0^T\dot w_k(t)^2\diff t \leq \alpha_k^2\left(4+\sqrt{2}+\left(\frac{\sinh(\beta)+\beta}{\sinh(\beta)-\beta}\right)^\frac{1}{2}\right)^2\int_0^Tu_k(t)^2\diff t\,.
    \end{equation*}
    Noting that $\alpha_k^2\geq\max(m,\frac{\beta^2}{T^2})$, we hence obtain the bounds
    \begin{align*}
        \norm{h_k}_{L^2(\muq)}^2 &\leq\frac{1}{30}\min\left(\frac{1}{m},\frac{T^2}{\beta^2}\right)\norm{H_k^a}_{L^2(\muq)}^2\,,\\
        \norm{\partial_th_k}_{L^2(\muq)}^2 &\leq \left(1+\frac{1}{\sqrt{3}}\right)^2\norm{H_k^a}_{L^2(\muq)}^2\,,\\
        \norm{\nabla h_k}_{L^2(\muq)}^2 &\leq\frac{1}{30}\norm{H_k^a}_{L^2(\muq)}^2\,,\\
        \norm{\nabla g_k}_{L^2(\muq)}^2 &\leq \min\left(\frac{1}{m},\frac{T^2}{\beta^2}\right)\left(1+\frac{1}{\sqrt{3}}\right)^2\norm{H_k^a}_{L^2(\muq)}^2\,,\\
        \norm{\partial_t\nabla g_k}_{L^2(\muq)}^2 &\leq \left(4+\sqrt{2}+\left(\frac{\sinh(\beta)+\beta}{\sinh(\beta)-\beta}\right)^\frac{1}{2}\right)^2\norm{H_k^a}_{L^2(\muq)}^2\,,\\
        \norm{Lg_k}_{L^2(\muq)}^2 &\leq\left(1+\frac{1}{\sqrt{3}}\right)^2\norm{H_k^a}_{L^2(\muq)}^2\,.
    \end{align*}
    and thus also
    \begin{equation*}
        \norm{\nabla^2g_k}_{L^2(\muq)}^2 \leq \left(1+\frac{1}{\sqrt{3}}\right)^2\left(1+\min\left(\frac{1}{m},\frac{T^2}{\beta^2}\right)\rho\right)\norm{H_k^a}_{L^2(\muq)}^2\,.
    \end{equation*}
    For a general function $f=\sum_{\alpha_k>\frac{\beta}{T}}b_kH_k^a$ in $\H_{h,a}$, we set
    \begin{equation*}
        h = \sum_{\alpha_k>\beta/T}b_kh_k\qquad\textup{and}\qquad g = \sum_{\alpha_k>\beta/T}b_kg_k
    \end{equation*}        
    to obtain $f = \partial_th-Lg$.
    By orthogonality of the functions $e_k$, the functions $h_k$ and $\partial_th_k$ are also orthogonal systems. Hence
    \begin{equation*}
        \norm{h}_{L^2(\muq)}^2 = \sum_{\alpha_k>\beta/T}\norm{h_k}^2,\qquad \norm{\partial_th}_{L^2(\muq)}^2 = \sum_{\alpha_k>\beta/T}\norm{\partial_th_k}^2
    \end{equation*}
    and
    \begin{align*}
        \norm{\nabla h}_{L^2(\muq)}^2 &= -(h,Lh)_{L^2(\muq)} = \sum_{\alpha_k>\beta/T}\alpha_k^2\norm{h_k}^2\\
        &= \sum_{\alpha_k>\beta/T}(h_k,Lh_k)_{L^2(\muq)} = \sum_{\alpha_k>\beta/T}\norm{\nabla h_k}_{L^2(\muq)}^2\,.
    \end{align*}
    The same is true for the functions $\nabla g_k$, $\partial_t\nabla g_k$ and $Lg_k$, so that analogous identities hold.
    Since $\norm{f}_{L^2(\muq)}^2 = \sum_{\alpha_k\geq\frac{\beta}{T}}b_k^2\norm{H_k^a}_{L^2(\muq)}^2$, this yields \eqref{eq:bound0order} and \eqref{eq:bound1order} with constants
    \begin{align*}
        c_0^{h,a} &= \left(\frac{1}{30}+\left(1+\frac{1}{\sqrt{3}}\right)^2\right)\min\left(\frac{1}{m},\frac{T^2}{\beta^2}\right)\qquad\textup{and}\\
        c_1^{h,a} &= \left(1+\frac{1}{\sqrt{3}}\right)^2\left(2+\min\left(\frac{1}{m},\frac{T^2}{\beta^2}\right)\rho\right)+ \left(4+\sqrt{2}+\left(\frac{\sinh(\beta)+\beta}{\sinh(\beta)-\beta}\right)^\frac{1}{2}\right)^2 + \frac{1}{30}\,.\\
    \end{align*}
    
    \emph{Case 4:} Let $f\in\H_{l,s}$. Similarly to the previous case, we now set $u_k(t) = e^{-\alpha_kt}+e^{-\alpha_k(T-t)}$, so that $H_k^s(t,x) = u_k(t)e_k(x)$. We again derive a representation
    \begin{equation*}
        H_k^s = \partial_th_k-Lg_k
    \end{equation*}
    for each of the basis functions $H_k^s$ with $k\in\N$ fixed by employing the ansatz 
    \begin{equation}
        u_k = \dot v_k+w_k\quad\textup{with }v_k(0)=v_k(T)=w_k(0)=w_k(T)=0\,.
    \end{equation}
    Then setting 
    \begin{equation*}
        h_k(t,x) = v_k(t)e_k(x)\qquad\textup{and}\qquad g_k(t,x) = \frac{1}{\alpha_k^2}w_k(t)e_k(x)
    \end{equation*}
    yields $H_k^s = \partial_th_k-Lg_k$. The only thing that changes compared to the previous case is the bound on $\int_0^T\dot u_k(t)^2\diff t$.  In this case, one has
    \begin{align*}
        \int_0^\frac{T}{2}\dot u_k(t)^2\diff t &= \alpha_k^2\int_0^T\left(-e^{-\alpha_kt}+e^{-\alpha_k(T-t)}\right)^2\diff t\\
        &= \alpha_ke^{-\alpha_kT}(\sinh(\alpha_kT)-\alpha_kT)\,, 
    \end{align*}
    whereas
    \begin{equation*}
        \int_0^\frac{T}{2}u_k(t)^2\diff t = \alpha_k^{-1}e^{-\alpha_kT}(\sinh(\alpha_kT)+\alpha_kT)\,.
    \end{equation*}
    Hence
    \begin{equation*}
        \int_0^T\dot u_k(t)^2\diff t = \alpha_k^2\frac{\sinh(\alpha_kT)-\alpha_kT}{\sinh(\alpha_kT)+\alpha_kT}\int_0^Tu_k(t)^2\diff t \leq \alpha_k^2\int_0^Tu_k(t)^2\diff t\,.
    \end{equation*}
    Again setting 
    \begin{equation*}
        h = \sum_{\alpha_k>\beta/T}b_kh_k\qquad\textup{and}\qquad g = \sum_{\alpha_k>\beta/T}b_kg_k
    \end{equation*}        
    yields $f = \partial_th-Lg$. As in the previous case, the bounds \eqref{eq:bound0order} and \eqref{eq:bound1order} follow with the constants
    \begin{align*}
        c_0^{h,s} &= \left(\frac{1}{30}+\left(1+\frac{1}{\sqrt{3}}\right)^2\right)\min\left(\frac{1}{m},\frac{T^2}{\beta^2}\right)\qquad\textup{and}\\
        c_1^{h,s} &= \left(1+\frac{1}{\sqrt{3}}\right)^2\left(2+\min\left(\frac{1}{m},\frac{T^2}{\beta^2}\right)\rho\right)+ \left(5+\sqrt{2}\right)^2 + \frac{1}{30}\,.\\
    \end{align*}

    For any $f\in L_0^2(\muq)$, combining the results on each of the subspaces using linearity, we can hence write $f=\partial_th-Lg$ with $h$ and $g$ satisfying \eqref{eq:hgDirichletcond} and \eqref{eq:gdomLcond}. Choosing $\beta=2$, the bounds \eqref{eq:bound0order} and \eqref{eq:bound1order} then follow with
    \begin{align*}
        c_0(T) &= 5\max\left(c_0^{l,a}, c_0^{l,s}, c_0^{h,a}, c_0^{h,s}, c_0^\bot\right) = 5c_0^{l,s}\\
        &\leq 2T^2+43\frac{1}{m}
    \end{align*}
    and
    \begin{align*}
        c_1(T) &= 5\max\left(c_1^{l,a}, c_1^{l,s}, c_1^{h,a}, c_1^{h,s}, c_1^\bot\right)=5\max\left(c_1^{l,s},c_1^{h,a},c_1^\bot\right)\\
        &\leq 290+\frac{991}{mT^2}+43\max\left(\frac{1}{m},\frac{T^2}{\pi^2}\right)\rho\,.
    \end{align*}

\end{proof}

\section{Optimal lifts on Riemannian manifolds with boundary}\label{sec:optimallifts}

    In the following, let $(M,g)$ be a complete connected and oriented Riemannian manifold with boundary and let
    \begin{equation*}
        \mu(\diff x)=\exp(-U(x))\,\nuM(\diff x)
    \end{equation*}
    with $U\in C^2(M)$ be a probability measure on $M$ such that \Cref{ass:U,ass:discretespec} are satisfied.

    The operator $(\frac{1}{2}L,\dom(L))$ defined by \Cref{cor:reillyestimate} is the generator in $L^2(\mu)$ of an overdamped Langevin diffusion on $M$ with reflection at the boundary $\dM$ and invariant probability measure $\mu$. Explicitly, the set of compactly supported smooth functions
    \begin{equation*}
        \D=\{u\in C_\c^\infty(M)\colon\partial_Nu=0\textup{ on }\dM\}
    \end{equation*}
    with Neumann boundary conditions is a core for the self-adjoint operator $(\frac{1}{2}L,\dom(L))$ and 
    \begin{equation*}
        Lf = -\nabla U\cdot\nabla f + \Delta f\qquad\textup{for }f\in\D\,.
    \end{equation*}
    The corresponding diffusion process solves the Stratonovich stochastic differential equation 
    \begin{equation*}
        \diff Z_t = -\frac{1}{2}\nabla U(Z_t)\diff t + \circ\diff B_t - N(Z_t)\diff L_t\,,
    \end{equation*}
    where $(B_t)_{t\geq0}$ is a Brownian motion on $M$ and $(L_t)_{t\geq 0}$ is the local time of $(Z_t)_{t\geq0}$ at the boundary $\dM$, see \cite{IkedaWatanabe1989SDE,Wang2014Analysis}.
    By \Cref{cor:reillyestimate}, the associated Dirichlet form is
    \begin{equation*}
        \Ecal(f,g) = \frac{1}{2}\int_M \langle\nabla f,\nabla g\rangle\diff\mu
    \end{equation*}
    with domain $\dom(\Ecal) = H^{1,2}(M,\mu)$.

\subsection{Lifts based on Hamiltonian dynamics on Riemannian manifolds}

    At first assume that $\dM=\emptyset$. We denote elements of the tangent bundle $\T M$ by $(x,v)$ with $x\in M$ and $v\in \T_xM$. Defining the Hamiltonian
    \begin{equation*}
        H(x,v) = U(x) + \frac{1}{2}\langle v,v\rangle_x
    \end{equation*}
    on $\T M$ yields the associated Hamiltonian flow $(X_t,V_t)_{t\geq 0}$ on $\T M$, see \cite[Section 3.7]{Abraham1987Mechanics}. It preserves the Hamiltonian $H$ and satisfies the equations
    \begin{align*}
        \frac{\diff}{\diff t}X_t = V_t\,,\qquad
        \frac{\nabla}{\diff t}V_t = -\nabla U(X_t)\,.
    \end{align*}
    Using horizontal and vertical lifts of vector fields, the tangent space $\T_{(x,v)}\T M$ of $\T M$ in $(x,v)\in\T M$ can be canonically identified with the direct sum
    \begin{equation}\label{eq:TTM}
        \T_{(x,v)}\T M = \T_xM\oplus\T_xM\,,
    \end{equation}
    see \cite[Section II.4]{Sakai1996Riemannian}. Using this identification, the vector field generating the Hamiltonian flow is $(x,v)\mapsto(v,-\nabla U(x))$. Furthermore, the Sasaki metric
    \begin{equation*}
        \tilde g_{(x,v)}\big((\zeta,\xi),(\zeta',\xi')\big) = g_x(\zeta,\zeta') + g_x(\xi,\xi')\quad\textup{for }(x,v)\in\T M,\,(\zeta,\xi), (\zeta',\xi')\in\T_{(x,v)}\T M
    \end{equation*}
    on $\T M$ turns the tangent bundle $(\T M,\tilde g)$ into a Riemannian manifold \cite{Sasaki1958Differential,Sakai1996Riemannian}. By Liouville's theorem, the Hamiltonian flow leaves the associated Riemannian volume form $\diff\nu_{\tilde g}$ on $\T M$ invariant. Therefore, the Hamiltonian flow preserves the probability measure $\hat\mu$ on $\T M$ with density proportional to $\exp(-H)$ with respect to the volume $\diff\nu_{\tilde g}$, which can be disintegrated as
    \begin{equation*}
        \hat\mu(\diff x\diff v) = \mu(\diff x)\kappa(x,\diff v)
    \end{equation*}
    with $\kappa(x,\cdot) = \mathcal{N}(0,g_x^{-1})$. The generator of the associated transition semigroup on $L^2(\hat\mu)$ is then given by 
    \begin{equation}\label{eq:hatL}
        \hat Lf(x,v) = \langle v,\nabla_xf(x,v)\rangle - \langle\nabla U(x),\nabla_vf(x,v)\rangle\,,
    \end{equation}
    and its domain satisfies
    \begin{equation*}
        \dom(\hat L)\supseteq\{f\in C^\infty(\T M)\cap L^2(\hat\mu)\colon\langle v,\nabla_xf\rangle - \langle\nabla U,\nabla_vf\rangle\in L^2(\hat\mu) \}\,.
    \end{equation*}
    \begin{lemm}\label{lem:HamDynlift}
        The Hamiltonian dynamics $(X_t,V_t)_{t\geq 0}$ is a second-order lift of the overdamped Langevin diffusion $(Z_t)_{t\geq 0}$.
    \end{lemm}
    \begin{proof}
        Note that it suffices to verify \eqref{eq:deflift0}--\eqref{eq:deflift2} on the core $\D$ for $(\frac{1}{2}L,\dom(L))$ defined in \eqref{eq:defD}. Hence let $f,g\in\D$. Then $\langle v,\nabla_x(f\circ\pi)\rangle = \langle v,\nabla f\circ\pi\rangle\in L^2(\hat\mu)$, so that $f\circ\pi\in\dom(\hat L)$ and \eqref{eq:deflift0} is satisfied. Furthermore,
        \begin{equation*}
            \int_{\T M}\hat L(f\circ\pi)(g\circ\pi)\diff\hat\mu = \int_M\int_{\T_x M}\langle v,\nabla f(x)\rangle\kappa(x,\diff v)\,g(x)\mu(\diff x) = 0\,,
        \end{equation*}
        so that \eqref{eq:deflift1} is satisfied. Finally, \eqref{eq:deflift2} is satisfied by
        \begin{align*}
            \frac{1}{2}\int_{\T M}(\hat L(f\circ\pi))^2\diff\hat\mu &= \frac{1}{2}\int_M\int_{\T_xM}\langle v,\nabla f(x)\rangle^2\kappa(x,\diff v)\mu(\diff x)\\
            &=\frac{1}{2}\int_M\langle\nabla f,\nabla f\rangle\diff\mu = \Ecal(f,f)\,.\qedhere
        \end{align*}
    \end{proof}

    Now let us consider a complete Riemannian manifold $(M,g)$ with boundary $\dM\neq\emptyset$. For $x\in\dM$ and $v\in\T_xM$ let 
    \begin{equation*}
        R_xv = v-2N(x)\langle N(x),v\rangle_x
    \end{equation*}
    be the specular reflection at the boundary. Note that, if $v$ is outward-pointing, then $R_xv$ is inward-pointing, and $R_xv$ preserves $\kappa(x,\diff v)$. The generator of Hamiltonian dynamics with reflection at the boundary is then given by the closure $(\hat L,\dom(\hat L))$ of $(\hat L,\mathcal{R})$ in $L^2(\hat\mu)$ with
    \begin{align*}
        \mathcal{R} = \{f\in C^\infty(\T M)\cap L^2(\hat\mu)\colon &\langle v,\nabla_xf\rangle - \langle\nabla U,\nabla_vf\rangle\in L^2(\hat\mu)\quad\textup{and}\\
        &f(x,R_xv)=f(x,v)\textup{ for all }x\in\dM\textup{ and }v\in\T_xM\}\,,
    \end{align*}
    and $\hat L$ given by \eqref{eq:hatL}. This definition clearly coincides with the definition above in case $\dM=\emptyset$. The associated process $(X_t,V_t)_{t\geq 0}$ follows the Hamiltonian flow on the interior of $M$, and the normal component of the velocity is reflected when hitting the boundary. 
    \begin{lemm}\label{eq:HamDynBdrylift}
        The Hamiltonian dynamics $(X_t,V_t)_{t\geq 0}$ with reflection at the boundary is a second-order lift of the overdamped Langevin diffusion $(Z_t)_{t\geq 0}$ with reflection at the boundary.
    \end{lemm}
    \begin{proof}
        Since $f\in\D$ implies $f\circ\pi\in\mathcal{R}$, the domain condition \eqref{eq:deflift0} is satisfied. The remaining conditions \eqref{eq:deflift1} and \eqref{eq:deflift2} follow as in \Cref{lem:HamDynlift}.
    \end{proof}

    \begin{lemm}[Adjoint of $\hat L$]\label{lem:Lhatstar}
        The adjoint of $(\hat L,\dom(\hat L))$ in $L^2(\hat\mu)$ satisfies $\dom(\hat L)\subseteq\dom(\hat L^*)$ and $\hat L^*g = -\hat Lg$ for all $g\in\dom(\hat L)$. 
    \end{lemm}
    \begin{proof}
        Let $f,g\in\mathcal{R}$. Integration by parts gives
        \begin{align*}
            \int_{\T M}\hat Lfg\diff\hat\mu = -\int_{\T M}f\hat Lg\diff\hat\mu + \int_{\dM}\int_{\T_x M}fg\langle v,N(x)\rangle\,\kappa(x,\diff v)\mu(\diff x)\,.
        \end{align*}
        Denoting $h(x,v) = f(x,v)\langle v,N(x)\rangle$ gives $h(x,R_xv) = -h(x,v)$, so that
        \begin{align*}
            \int_{\T_xM}h(x,v)g(x,v)\kappa(x,\diff v) &= \int_{\T_xM}h(x,R_xv)g(x,R_xv)\kappa(x,\diff v)\\
            &=-\int_{\T_xM}h(x,v)g(x,v)\kappa(x,\diff v)
        \end{align*}
        since $R_x$ leaves $\kappa(x,\cdot)$ invariant, so that the boundary terms vanish. Hence the operator $(-\hat L^*,\dom(\hat L^*))$ extends $(\hat L,\mathcal{R})$ and thus also $(\hat L,\dom(\hat L))$.
    \end{proof}

    \emph{Randomised Riemannian Hamiltonian Monte Carlo} with refresh rate $\gamma>0$ is obtained by interspersing Hamiltonian dynamics with complete velocity refreshments according to $\kappa(x,\diff v)$ after random times with exponential distribution $\mathrm{Exp}(\gamma)$.
    By successively conditioning on the refreshment times, one can show that randomised Riemannian Hamiltonian Monte Carlo again leaves $\hat\mu$ invariant \cite{Bou-Rabee2017RHMC}, and its generator in $L^2(\hat\mu)$ is given by
    \begin{equation*}
        \hat L^{(\gamma )}_\RHMC= \hat L+\gamma(\Pi-I)\qquad\textup{and}\qquad\dom\big(\hat L^{(\gamma)}_\RHMC\big) = \dom(\hat L)\,,
    \end{equation*}
    where
    \begin{equation*}
        (\Pi f)(x,v)=\int_{\R^d}f(x,w)\kappa(x,\diff w)\,.
    \end{equation*}

    \emph{Riemannian Langevin dynamics} is obtained by adding Ornstein-Uhlenbeck noise instead of discrete jumps in the velocity, i.e.\ it is the solution $(X_t,V_t)_{t\geq 0}$ to the stochastic differential equation
    \begin{align*}
        \diff X_t &= V_t\diff t\,,\\
        \diff V_t & = -\nabla U(X_t)\diff t - \gamma V_t\diff t + \sqrt{2\gamma}\circ\diff B_t + \diff K_t\,,
    \end{align*}
    where
    \begin{equation*}
        K_t = -2\sum_{0< s\leq t}\langle N(X_s),V_{s-}\rangle N(X_s)\cdot 1_{\dM}(X_s)
    \end{equation*}
    is a confinement process modelling specular reflections the boundary. Well-posedness of this equation and existence of solutions, in particular non-accumulation of boundary hits, is considered in \cite{BossyJabir2011Confined,BossyJabir2015Lagrangian} in the Euclidean case, and the associated kinetic Fokker-Planck equation with specular boundary conditions is considered in \cite{Dong2022Specular}.  The invariant probability measure of $(X_t,V_t)_{t\geq 0}$ is $\hat\mu$, see \cite{GrothausMertin2022LangevinManifolds}. We denote the associated generator in $L^2(\hat\mu)$ by $(\hat L^{(\gamma)}_\LD,\dom(\hat L^{(\gamma)}_\LD))$.
    Let 
    \begin{equation}\label{eq:Lvdef}
        L_vf = -\langle v,\nabla_vf\rangle+\Delta_vf\,,
    \end{equation}
    where $\Delta = \Delta_x+\Delta_v$ is the splitting of the Laplace-Beltrami operator $\Delta$ on $\T M$ associated to the decomposition \eqref{eq:TTM}, i.e.\ $\Delta_v = \divg_v\nabla_v$ is the Laplace-Beltrami operator in $v$. Then,
    for functions $f\in\mathcal{R}$ with $L_vf\in L^2(\hat\mu)$, the generator is given by
     \begin{equation}\label{eq:hatLgamma}
        \hat L^{(\gamma)}_\LD f = \hat Lf + \gamma L_vf\,.
    \end{equation}

    \begin{lemm}
        For any choice of $\gamma>0$, both $(\hat L_\RHMC^{(\gamma)},\dom(\hat L_\RHMC^{(\gamma)}))$ and $(\hat L_\LD^{(\gamma)},\dom(\hat L_\LD^{(\gamma)}))$ are lifts of $(\frac{1}{2}L,\dom(L))$.
    \end{lemm}
    \begin{proof}
        Firstly, since $\D$ forms a core for $(L,\dom(L))$, the condition \eqref{eq:deflift0} is satisfied for both $\hat L_\RHMC^{(\gamma)}$ and $\hat L^{(\gamma)}$. Secondly, both generators are obtained from $\hat L$ by adding a term that only acts on the velocity and leaves $\kappa(x,\cdot)$ invariant. This means that
        \begin{equation*}
            \hat L_\RHMC^{(\gamma)}(f\circ\pi) = \hat L^{(\gamma)}_\LD(f\circ\pi) = \hat L(f\circ\pi)
        \end{equation*}
        for any $f\in\D$, so that \eqref{eq:deflift1} and \eqref{eq:deflift2} are immediately satisfied. 
    \end{proof}

    \begin{rema}
        Suppose that $M\subseteq\R^d$ is closed and convex with smooth boundary, so that $M$ is a Riemannian manifold with boundary when equipped with the flat Euclidean metric, as in \Cref{rem:Mconvex}. The construction of RHMC and Langevin dynamics on $M$ then reduces to the usual processes on $\R^d$ constrained to $M$ via a specular reflection at the boundary. In case $U$ is constant, the invariant probability measure $\hat\mu$ is the product of the uniform distribution on $M$ and a standard normal distribution in the velocity. In this case, the Hamiltonian dynamics on $M$ reduces to billiards in a convex set, see \cite{Tabachnikov2005Billiards}. 
    \end{rema}

\subsection{Space-time Poincaré inequality on Riemannian manifolds}
    
    Albritton, Armstrong, Mourrat and Novack \cite{Albritton2021Variational} developed an approach to proving exponential decay for the kinetic Fokker-Planck equation using space-time Poincaré inequalities. Quantitative upper bounds for various non-reversible Markov processes were then obtained by Cao, Lu and Wang \cite{Lu2022PDMP,Cao2019Langevin} by proving such space-time Poincaré inequalities for the Kolmogorov backward equations associated to these processes. In \cite{EberleLoerler2024Lifts}, the authors show how the framework of second-order lifts together with a divergence lemma as proved in \Cref{sec:divergence} can be used to prove such space-time Poincaré inequalities and thereby obtain quantitative bounds on the relaxation time of non-reversible Markov processes. In this section, we prove a space-time Poincaré inequality for RHMC and Langevin dynamics on a Riemannian manifold $M$ with reflection at the boundary.

    In the following, we fix $T\in(0,\infty)$ and write $L^2(\muqk)=L^2([0,T]\times\T M,\muqk)$. We consider the operator $(A,\dom(A))$ on $L^2(\muqk)$ defined by
    \begin{equation}
        \label{def:A} Af=-\partial_tf+\hat Lf
    \end{equation}
    with domain consisting of all functions  
    $ f\in L^2(\muqk)$ such that $f(\cdot,x,v)$ is absolutely continuous for $\hat\mu$-a.e.\ $(x,v)\in \T M$ with $\partial_t f\in L^2(\muqk)$, and $f(t,\cdot)\in\dom(\hat L)$ for $\lambda\textup{-a.e.\ }t\in[0,T]$ with $\hat Lf\in L^2(\muqk)$. Here $\hat L$ is the generator \eqref{eq:hatL} of the Hamiltonian flow.

    We consider the adjoint $(A^*,\dom(A^*))$ in $L^2(\muqk)$ of the operator $(A,\dom(A))$. An integration by parts and \Cref{lem:Lhatstar} yield
    \begin{align*}
        (Af,g)_{L^2(\muqk)} &= (f,-Ag)_{L^2(\muqk)} + \int_{\T M}(f(T,\cdot)g(T,\cdot)-f(0,\cdot)g(0,\cdot))\diff\hat\mu
    \end{align*}
    for all $f,g\in\dom(A)$. Thus functions
    $g\in\dom(A)$ satisfying $g(0,\cdot)=g(T,\cdot)=0$ are contained in the domain of $A^*$, and
    \begin{equation}\label{Astarg}
        A^*g=-Ag=\partial_tg-\hat Lg\, . 
    \end{equation}
    In the following, for $(t,x,v)\in [0,\infty )\times \T M$ we set $\pi(t,x,v)=(t,x)$, whereas $\pi(x,v)=x$.
    
    \begin{rema}\label{rem:hgdomAstar}
        For $h$ and $g$ as in \Cref{thm:divergence} and $(x,v)\in\T M$ with $x\in\dM$, we have 
        \begin{align*}
            (h\circ\pi+\hat L(g\circ\pi))(\cdot,x,R_xv) &= h(\cdot,x)+\big\langle v-2N(x)\langle N(x),v\rangle),\nabla_xg(\cdot,x)\big\rangle \\
            &= (h\circ\pi+\hat L(g\circ\pi))(\cdot,x,v)\,,
        \end{align*}
        since $\langle N,\nabla_xg\rangle = \partial_Ng = 0$ on $\dM$, and
        \begin{equation*}
            (h\circ\pi+\hat L(g\circ\pi))(0,\cdot) = (h\circ\pi+\hat L(g\circ\pi))(T,\cdot) = 0\,.
        \end{equation*}
        so that $h\circ\pi+\hat L(g\circ\pi) \in\dom(A^*)$.
    \end{rema}

    \begin{lemm}
        \label{lem:Astar}
        Let $X=\begin{pmatrix}-h\\\nabla_xg\end{pmatrix}$ as in \Cref{thm:divergence} and $x\in M$. Then 
        \begin{eqnarray}\label{eq:Astar1}
            \int_{\T_xM} A^*(h\circ\pi+\hat L(g\circ\pi))\kappa(x,\diff v) &=& \overline{\nabla}^*X\circ\pi \, ,\quad\text{ and}\\
            \norm{A^*(h\circ\pi+\hat L(g\circ\pi))-\overline\nabla^*X\circ\pi}_{L^2(\muqk)}
            &\leq &\sqrt{2}\norm{\overline{\nabla}X}_{L^2(\muq)}\,.\label{eq:Astar2}
        \end{eqnarray}
    \end{lemm} 
    \begin{proof}
        The proof in the Riemannian setting works as in \cite[Lemma 22]{EberleLoerler2024Lifts}.
    \end{proof}

    We are now ready to state and prove the space-time Poincaré inequality that is the key ingredient for deriving the relaxation time of lifts.

    For any $x\in M$ we let $H^{-1}(\kappa_x)$ denote the dual of the Gaussian Sobolev space $H^{1,2}(\kappa_x) = H^{1,2}(\T_xM,\kappa(x,\cdot))$. Note that
    \begin{equation*}
        \norm{f}_{H^{1,2}(\kappa_x)}^2 = \norm{f}_{L^2(\kappa_x)}^2 + \norm{\nabla f}_{L^2(\kappa_x)}^2 = \big(f,(I-L_v)f\big)_{L^2(\kappa_x)}\,,
    \end{equation*}
    where $L_v$ is given by \eqref{eq:Lvdef}. For $f\in L^2(\kappa_x)$, the $H^{-1}(\kappa_x)$-norm is then 
    \begin{equation*}
        \norm{f}_{H^{-1}(\kappa_x)}^2 = \big(f,(I-L_v)^{-1}f\big)_{L^2(\kappa_x)}\,,
    \end{equation*}
    where the inverse exists by the spectral decomposition. In particular, a spectral decomposition shows that for $f\in H^{2,2}(\kappa_x)$, 
    \begin{equation}\label{eq:H-1comparison1}
        \norm{L_vf}_{H^{-1}(\kappa_x)}^2 = \big(L_vf,(I-L_v)^{-1}L_vf\big)_{L^2(\kappa_x)} \leq (f,-L_vf)_{L^2(\kappa_x)}
    \end{equation}
    and
    \begin{equation}\label{eq:H-1comparison2}
        \norm{L_vf}_{H^{-1}(\kappa_x)}^2 \geq \frac{1}{2}(f,-L_vf)_{L^2(\kappa_x)}\,.
    \end{equation}
    Let $L^2(\lambda\otimes\mu;H^{1,2}(\kappa))$ denote the subspace of functions $f\in L^2(\muqk)$ such that $f(t,x,\cdot)\in H^{1,2}(\kappa_x)$ for $\lambda\otimes\mu$-a.e.\ $(t,x)\in[0,T]\times M$ and
    \begin{equation*}
        \norm{f}_{L^2(\lambda\otimes\mu;H^{1,2}(\kappa))}^2=\int_{[0,T]\times M}\norm{f(t,x,\cdot)}_{H^{1,2}(\kappa_x)}^2\mu(\diff x)\lambda(\diff t)<\infty\,.
    \end{equation*}
    Its dual is the similarly defined $L^2(\lambda\otimes\mu;H^{-1}(\kappa))$. In particular, for $f\in L^2(\lambda\otimes\mu;H^{-1}(\kappa))$ and $g\in L^2(\lambda\otimes\mu,H^{1,2}(\kappa))$,
    \begin{equation*}
        (f,g)_{L^2(\muqk)} \leq \norm{f}_{L^2(\lambda\otimes\mu;H^{-1}(\kappa))} \norm{g}_{L^2(\lambda\otimes\mu;H^{1,2}(\kappa))}\,.
    \end{equation*}

    \begin{theo}[Space-time Poincaré inequality]\label{thm:STPI}
        For any $T>0$, there exist constants $C_0(T),C_1(T)>0$ such that for all functions $f\in\dom(A)$,
        \begin{equation}\label{eq:STPI1}
            \norm{f-\int f\diff(\muqk)}_{L^2(\muqk)}^2\leq C_0(T)\norm{Af}_{L^2(\muqk)}^2+C_1(T)\norm{(\Pi-I)f}_{L^2(\muqk)}^2
        \end{equation}
        and
        \begin{equation}\label{eq:STPI2}
            \norm{f-\int f\diff(\muqk)}_{L^2(\muqk)}^2\leq 2C_0(T)\norm{Af}_{L^2(\lambda\otimes\mu,H^{-1}(\kappa))}^2+C_1(T)\norm{(\Pi-I)f}_{L^2(\muqk)}^2\,,
        \end{equation}
        where
        \begin{align*}
            C_0(T) &= 2c_0(T) = 4T^2 + 86\frac{1}{m}\,,\\
            C_1(T) &= 3+4c_1(T) = 1163+\frac{3964}{mT^2}+172\max\left(\frac{1}{m},\frac{T^2}{\pi^2}\right)\rho\,.
        \end{align*}
    \end{theo}

    The proof of \eqref{eq:STPI1} is an adaptation of \cite[Theorem 23]{EberleLoerler2024Lifts} to the Riemannian setting with boundary. An expression similar to \eqref{eq:STPI2} has first been shown in the Euclidean setting in \cite{Cao2019Langevin} but with non-optimal constants. The treatment of Langevin dynamics requires the $L^2(\lambda\otimes\mu;H^{-1}(\kappa))$-norm of $Af$ in \eqref{eq:STPI2}.

    In the proof we apply the following \emph{space-time property of lifts}, see \cite[Lemma 18]{EberleLoerler2024Lifts}. It states that for any $f,g,h\in L^2(\lambda\otimes\mu)$ such that $f\circ\pi\in\dom(A)$ and $g(t,\cdot)\in\dom(L)$ for a.e.\ $t\in[0,T]$ with $Lg\in L^2(\lambda\otimes\mu)$, we have
    \begin{equation}\label{eq:xtproplift}
        \frac{1}{2}\left(A(f\circ\pi),h\circ\pi+\hat L(g\circ\pi)\right)_{L^2(\muqk)}=-\frac{1}{2}(\partial_tf,h)_{L^2(\lambda\otimes\mu)}+\Ecal_T(f,g)\, ,
    \end{equation}
    where $\Ecal_T(f,g)=\int_0^T\Ecal(f(t,\cdot),g(t,\cdot ))\diff t$ is the time-integrated Dirichlet form associated to $(\frac{1}{2}L,\dom(L))$.

    \begin{proof}[Proof of \Cref{thm:STPI}]
        Fix a function $f_0\in\dom(A)$, let $f=f_0-\int f_0\diff(\muqk)$, and let $\tilde f\colon[0,T]\times\R^d\to\R$ be such that $\Pi f=\tilde f\circ\pi$.
        First note that
        \begin{align}
            \norm{f}_{L^2(\muqk)}^2&=\norm{f-\Pi f}_{L^2(\muqk)}^2+\norm{\tilde f}_{L^2(\muq)}^2\,,\label{L2decompf}
        \end{align}
        so that it suffices to bound $\norm{\tilde f}_{L^2(\muq)}^2$ by $\norm{f-\Pi f}_{L^2(\muqk)}^2$ and $\norm{Af}_{L^2(\muqk)}^2$ to obtain \eqref{eq:STPI1}, and by $\norm{f-\Pi f}_{L^2(\muqk)}^2$ and $\norm{Af}_{L^2(\muq,H^{-1}(\kappa))}^2$ to obtain \eqref{eq:STPI2}. 
        
        By \Cref{thm:divergence} applied to $\tilde f$, there
        exist functions $h\in H^{1,2}_\DC(\muq)$ and $g\in H^{2,2}_\DC(\muq)$ satisfying \eqref{eq:gdomLcond} such that 
        \begin{equation*}
            \tilde f = \partial_th-Lg\,.
        \end{equation*}
        In particular,
        \begin{equation}\label{L2tildef}
            \norm{\tilde f}_{L^2(\muq)}^2=(\tilde f,\tilde f)_{L^2(\muq)}=(\tilde f,\partial_th-Lg)_{L^2(\muq)}\,.
        \end{equation}
        Using the space-time property of lifts \eqref{eq:xtproplift}, we see that
        \begin{align*}
            (\tilde f,\partial_th-Lg)_{L^2(\muq)}&=\big({-\partial_t\tilde f},h\big)_{L^2(\muq)}+2\Ecal_T(\tilde f,g)\\
            &=\left(A(\tilde f\circ\pi),h\circ\pi+\hat L(g\circ\pi)\right)_{L^2(\muqk)}\\
            &=\big(Af,h\circ\pi+\hat L(g\circ\pi)\big)_{L^2(\muqk)}\\
            &\qquad+\big(A(\Pi f-f),h\circ\pi+\hat L(g\circ\pi)\big)_{L^2(\muqk)}\,.
        \end{align*}
        Let us consider the two summands separately.
    
        Since $h\circ\pi+\hat L(g\circ\pi)\in L^2(\lambda\otimes\mu;H^{1,2}(\kappa))$, the first summand can be estimated as
        \begin{align*}
            \left(Af,h\circ\pi+\hat L(g\circ\pi)\right)_{L^2(\muqk)}&\leq \norm{Af}_{L^2(\muqk)}\norm{h\circ\pi+\hat L(g\circ\pi)}_{L^2(\muqk)}\\
            &\leq c_0(T)^\frac{1}{2}\norm{Af}_{L^2(\muqk)}\norm{\tilde f}_{L^2(\muq)}\, ,
        \end{align*}
        where we used that 
        \begin{equation*}
            \norm{h\circ\pi+\hat L(g\circ\pi)}_{L^2(\muqk)}^2 = \norm{h}_{L^2(\muq)}^2 + \norm{\nabla g}_{L^2(\muq)}^2\leq c_0(T)\norm{\tilde f}_{L^2(\muq)}^2
        \end{equation*}
        by \eqref{eq:bound0order}.

        By \Cref{rem:hgdomAstar}, $h\circ\pi+\hat L(g\circ\pi)\in\dom(A^*)$, so that the second summand satisfies
        \begin{equation*}
            \left(A(\Pi f-f),h\circ\pi+\hat L(g\circ\pi)\right)_{L^2(\muqk)}=\left(\Pi f-f,A^*(h\circ\pi+\hat L(g\circ\pi))\right)_{L^2(\muqk)},
        \end{equation*}
        Since $\tilde f= \overline\nabla^*X$,
        \Cref{lem:Astar} and \Cref{thm:divergence} yield
        \begin{align*}
            \norm{A^*(h\circ\pi+\hat L(g\circ\pi))}_{L^2(\muqk)}^2
            &=\norm{\tilde f}_{L^2(\muq)}^2+\norm{A^*(h\circ\pi+\hat L(g\circ\pi))-\overline\nabla^*X\circ\pi}_{L^2(\muqk)}^2\,\\
            &\leq \norm{\tilde f}_{L^2(\muq)}^2+2\norm{\overline\nabla X}_{L^2(\muq)}^2\leq(1+2c_1(T))\norm{\tilde f}_{L^2(\muq)}^2\,,
        \end{align*}
        so that the second summand can be estimated as
        \begin{equation*}
            \left(A(\Pi f-f),h\circ\pi+\hat L(g\circ\pi)\right)_{L^2(\muqk)}\leq \left(1+2c_1(T)\right)^\frac{1}{2}\norm{f-\Pi f}_{L^2(\muqk)}\norm{\tilde f}_{L^2(\muqk)}\,.
        \end{equation*}
    
        Putting things together, we hence obtain by \eqref{L2tildef},
        \begin{align*}
            \norm{\tilde f}_{L^2(\muq)}&\leq c_0(T)^\frac{1}{2}\norm{Af}_{L^2(\muqk)}+\left(1+2c_1(T)\right)^\frac{1}{2}\norm{f-\Pi f}_{L^2(\muqk)}
        \end{align*}
        so that by \eqref{L2decompf},
        \begin{align*}
            \norm{f}_{L^2(\muqk)}^2
            &\leq2c_0(T)\norm{Af}_{L^2(\muqk)}^2+(3+4c_1(T))\norm{f-\Pi f}_{L^2(\muqk)}^2\,
        \end{align*}
        which proves \eqref{eq:STPI1}.

        Similarly, to obtain \eqref{eq:STPI2}, the first summand can be estimated as
        \begin{align*}
            \left(Af,h\circ\pi+\hat L(g\circ\pi)\right)_{L^2(\muqk)}&\leq \norm{Af}_{L^2(\muq,H^{-1}(\kappa))}\norm{h\circ\pi+\hat L(g\circ\pi)}_{L^2(\muq,H^{1,2}(\kappa))}\\
            &\leq (2c_0(T))^\frac{1}{2}\norm{Af}_{L^2(\muq),H^{-1}(\kappa))}\norm{\tilde f}_{L^2(\muq)}\, ,
        \end{align*}
        where we used that
        \begin{align*}
            \norm{h\circ\pi+\hat L(g\circ\pi)}_{L^2(\muq,H^{1,2}(\kappa))}^2 &= \norm{h\circ\pi+\hat L(g\circ\pi)}_{L^2(\muqk)}^2 + \norm{\nabla_x(h\circ\pi+\hat L(g\circ\pi))}_{L^2(\muqk)}^2\\
            &=\norm{h}_{L^2(\muq)}^2+\norm{\nabla g}_{L^2(\muq)}^2 + \norm{v\cdot\nabla g}_{L^2(\muqk)}^2\\
            &\leq 2c_0(T)\norm{\tilde f}_{L^2(\muq)}^2
        \end{align*}
        and $c_0(T)$ is the constant from \eqref{eq:bound0order}. One then concludes analogously.
    \end{proof}
    
    As demonstrated in \cite{EberleLoerler2024Lifts}, such a space-time Poincaré inequality yields the following exponential decay in $T$-average of the associated semigroup. 

    \begin{theo}[Exponential Decay]\label{thm:decay}
    Let $\gamma>0$ and $T>0$ and set
    \begin{equation*}
        \nu = \frac{\gamma}{\gamma^2C_0(T)+C_1(T)}\,.
    \end{equation*}
    \begin{enumerate}[(i)]
        \item The transition semigroup $(\hat P_t^{(\gamma)})_{t\geq0}$ of Riemannian RHMC with reflection at the boundary and refresh rate $\gamma$ satisfies
        \begin{equation*}
            \frac{1}{T}\int_t^{t+T}\norm{\hat P_s^{(\gamma)}f}_{L^2(\hat\mu)}\diff s\leq e^{-\nu t}\norm{f}_{L^2(\hat\mu)}
        \end{equation*}
        for all $f\in L^2_0(\hat\mu)$.

        \item The transition semigroup $(\hat P_t^{(\gamma)})_{t\geq0}$ of Riemannian Langevin dynamics with reflection at the boundary and friction $\gamma$  satisfies
        \begin{equation*}
            \frac{1}{T}\int_t^{t+T}\norm{\hat P_s^{(\gamma)}f}_{L^2(\hat\mu)}\diff s\leq e^{-\frac{1}{2}\nu t}\norm{f}_{L^2(\hat\mu)}
        \end{equation*}
        for all $f\in L^2_0(\hat\mu)$.

    \end{enumerate}
    \end{theo}
    \begin{proof}
    \begin{enumerate}[(i)]
        \item The proof in this case uses \eqref{eq:STPI1} and is identical to that of \cite[Theorem 17]{EberleLoerler2024Lifts}.
        \item Let $f\in L_0^2(\hat\mu)$. Then by \eqref{eq:hatLgamma} and the antisymmetry of $\hat L$,
        \begin{align*}
            \MoveEqLeft\frac{\diff}{\diff t}\int_t^{t+T}\norm{\hat P_s^{(\gamma)}f}_{L^2(\hat\mu)}^2\diff s=\norm{\hat P_{t+T}^{(\gamma)}f}_{L^2(\hat\mu)}^2-\norm{\hat P_{t}^{(\gamma)}f}_{L^2(\hat\mu)}^2\\
            &=\int_t^{t+T}\frac{\diff}{\diff s}\norm{\hat P_s^{(\gamma)}f}_{L^2(\hat\mu)}^2\diff s = 2\int_t^{t+T}\left(\hat P_s^{(\gamma)}f,(\hat L+\gamma L_v)\hat P_s^{(\gamma)}f\right)_{L^2(\hat\mu)}\diff s\\
            &=-2\gamma\int_t^{t+T}\left(\hat P_s^{(\gamma)}f,-L_v\hat P_s^{(\gamma)}f\right)_{L^2(\hat\mu)}\diff s \leq-2\gamma\norm{L_v\hat P_{t+\cdot}^{(\gamma)}f}_{L^2(\lambda\otimes\mu,H^{-1}(\kappa))}^2\,,
        \end{align*}
        where we used \eqref{eq:H-1comparison1} in the last step. Since $\hat P_{t+\cdot}^{(\gamma)} f\in\dom(A)$ and $A\hat P_{t+\cdot}^{(\gamma)} f=\gamma L_v\hat P_{t+\cdot}^{(\gamma)} f$, the space-time Poincaré inequality \eqref{eq:STPI2} yields
        \begin{equation*}
            \norm{\hat P_{t+\cdot}^{(\gamma)} f}_{L^2(\muqk)}^2\leq 2\gamma^2C_0(T)\norm{L_v\hat P_{t+\cdot}^{(\gamma)} f}_{L^2(\lambda\otimes\mu,H^{-1}(\kappa))}^2 + C_1(T)\norm{(\Pi -I)\hat P_{t+\cdot}^{(\gamma)}f}_{L^2(\muqk)}^2\,.
        \end{equation*}
        The Gaussian Poincaré inequality and \eqref{eq:H-1comparison2} yield
        \begin{equation*}
            \norm{(\Pi -I)\hat P_{t+\cdot}^{(\gamma)}f}_{L^2(\muqk)}^2\leq 2\norm{L_v\hat P_{t+\cdot}^{(\gamma)} f}_{L^2(\muq,H^{-1}(\kappa))}^2\,,
        \end{equation*}
        so that
        \begin{equation*}
            \norm{\hat P_{t+\cdot}^{(\gamma)} f}_{L^2(\muqk)}^2\leq (2\gamma^2C_0(T)+2C_1(T))\norm{L_v\hat P_{t+\cdot}^{(\gamma)} f}_{L^2(\lambda\otimes\mu,H^{-1}(\kappa))}^2\,.
        \end{equation*}
        Therefore, Grönwall's inequality yields
        \begin{align*}
            \frac{1}{T}\int_t^{t+T}\norm{\hat P_s^{(\gamma)}f}_{L^2(\hat\mu)}^2\diff s&\leq e^{-\nu t}\frac{1}{T}\int_0^T\norm{\hat P_s^{(\gamma)}f}_{L^2(\hat\mu)}^2\diff s\leq e^{-\nu t}\norm{f}_{L^2(\hat\mu)}^2\,,
        \end{align*}
        and we obtain the exponential contractivity in $T$-average with rate $\frac{1}{2}\nu$ by the Cauchy-Schwarz inequality.
    \end{enumerate}
    \end{proof}

    \subsection{Optimality of randomised Hamiltonian Monte Carlo and Langevin dynamics on Riemannian manifolds with boundary}

    The exponential decay of $T$-averages in \Cref{thm:decay} immediately yields the bounds
    \begin{equation*}
        t_\rel(\hat P^{(\gamma)})\leq\frac{1}{\nu}+T\qquad\textup{or}\qquad t_\rel(\hat P^{(\gamma)})\leq\frac{2}{\nu}+T
    \end{equation*}
    for the relaxation times of Riemannian RHMC with reflection at the boundary and refresh rate $\gamma$, as well as Riemannian Langevin dynamics with reflection at the boundary and friction $\gamma$, respectively, see \cite[Lemma 8]{EberleLoerler2024Lifts}. For a fixed $T>0$, the rate $\nu$ is maximised by choosing $\gamma = \sqrt{{C_1(T)}/{C_0(T)}}$,
    and this choice yields the upper bounds
    \begin{equation*}
        t_\rel(\hat P^{(\gamma)})\leq2\sqrt{C_0(T)C_1(T)}+T\qquad\textup{and}\qquad t_\rel(\hat P^{(\gamma)})\leq4\sqrt{C_0(T)C_1(T)}+T\,,
    \end{equation*}
    respectively.
    To obtain explicit bounds, we set $T=\frac{\pi}{\sqrt{m}}$, which yields
    \begin{align}\label{eq:C0C1explicit}
        C_0(T)= (4\pi^2+86)\frac{1}{m}\qquad\textup{and}\qquad C_1(T) = 1163+\frac{3964}{\pi^2}+172\frac{\rho}{m}\,.
    \end{align}

    \begin{coro}[$C$-optimality of RHMC and Langevin dynamics]\label{cor:Coptimal}
        Let $c\geq 0$. On the class of potentials satisfying \Cref{ass:U} with $\rho\leq cm$ and \Cref{ass:discretespec}, Riemannian randomised Hamiltonian Monte Carlo with reflection at the boundary with refresh rate 
        \begin{equation}\label{eq:gammaopt}
            \gamma = \sqrt{\frac{(1163+\frac{3964}{\pi^2})m+172\rho}{4\pi^2+86}}
        \end{equation}
        is a $C$-optimal lift of the corresponding overdamped Langevin diffusion with reflection at the boundary with
        \begin{equation*}
            C = 2\left(887\sqrt{1+c/9}+\pi\right)\,.
        \end{equation*}

        Similarly, on the same class of potentials, Langevin dynamics with reflection at the boundary with friction $\gamma$ given by \eqref{eq:gammaopt} is a $C$-optimal lift of the same process with 
        \begin{equation*}
            C = 2\left(1773\sqrt{1+c/9}+\pi\right)\,.
        \end{equation*}

        More generally, for $A\in[1,\infty)$, Riemannian RHMC is a $C(A)$-optimal lift with
        \begin{equation*}
            C(A)=  \left(3381+344c\right)A+2\pi
        \end{equation*}
        for any $\gamma$ satisfying $1/A\le \frac{\gamma}{\sqrt{m}}\le A$. Similarly, for the same choices of $\gamma$, Riemannian Langevin dynamics is a $C(A)$-optimal lift with
        \begin{equation*}
            C(A) = \left(6761+344c\right)A+2\pi \,.
        \end{equation*}
    \end{coro}
    \begin{proof}
        The relaxation time of the corresponding overdamped Langevin diffusion with reflection at the boundary is $\frac{2}{m}$. Choosing $T=\frac{\pi}{\sqrt{m}}$ and plugging \eqref{eq:C0C1explicit} into the expression for $\nu$ gives
        \begin{align*}
            \frac{1}{\nu}&=2\sqrt{(4\pi^2 + 86)\left(1163+\frac{3964}{\pi^2}\right)}\frac{1}{\sqrt{m}}\sqrt{1+\frac{172}{1163+\frac{3964}{\pi^2}}\frac{\rho}{m}}\,,
        \end{align*}
        so that 
        \begin{equation*}
            t_\rel(\hat P^{(\gamma)})\leq \frac{1}{\sqrt{m}}\left(887\sqrt{1+\frac{\rho}{9m}}+\pi\right)\quad\textup{and}\quad t_\rel(\hat P^{(\gamma)})\leq \frac{1}{\sqrt{m}}\left(1773\sqrt{1+\frac{\rho}{9m}}+\pi\right)
        \end{equation*}
        for RHMC and Langevin dynamics, respectively.
        The final two statements follow similarly from the expression
        \begin{equation*}
            \frac{1}{\nu} = \gamma C_0(T) + \frac{1}{\gamma}C_1(T) \leq \sqrt{m}AC_0(T) + \frac{A}{\sqrt{m}}C_1(T)\,.\qedhere
        \end{equation*}
    \end{proof}
    \begin{rema}
        The explicit values of the constants in the divergence lemma \Cref{thm:divergence} are surely not optimal, so that we expect \Cref{cor:Coptimal} to hold with better constants. Nonetheless, in the case of mildly negative Bakry-\'Emery curvature $\Ric+\nabla^2U$, i.e.\ if the lower bound $\rho$ in \Cref{ass:U} \eqref{ass:RicHess} is of the same order as the inverse Poincar\'e constant $m$, the result shows that Riemannian RHMC and Riemannian Langevin dynamics with reflection at the boundary are close to being optimal lifts.
    \end{rema}

\appendix

\section{Appendix}\label{sec:appendix}
    In this appendix, for the convenience of the reader, we recall some concepts from Riemannian geometry and fix notation. Particular emphasis is placed on manifolds with boundary and curvature properties of the boundary, measured in terms of the second fundamental form. We finally provide a self-contained proof for the generalisation of the Reilly formula stated in \Cref{lem:GenReilly}. 

    We mainly follow \cite{Lee2013Smooth} and \cite{Lee2018Riemannian} as well as \cite{Gallot2004Riemannian}, and point to these sources for a more comprehensive overview of the topic.
    
    \subsection{Manifolds with boundary and manifolds with corners}\label{ssec:ManifoldBdry}
    
    A \emph{$d$-dimensional topological manifold with boundary} is a second-countable Hausdorff space such that every point has a neighbourhood that is homeomorphic to an open subset of $\R^n$ or a relatively open subset of $\R^d_+ = \{(x_1,\dots,x_d)\in\R^d\colon x_d\geq0\}$. A chart $(U,\phi)$ is an open subset $U\subset M$ together with a map $\phi\colon U\to\R^d$ that is a homeomorphism onto an open subset of $\R^d$ or $\R^d_+$. It is called a \emph{boundary chart} if $\phi(U)$ is an open subset of $\R_+^d$ such that $\phi(U)\cap\partial\R_+^d\neq\emptyset$ and an \emph{interior chart} otherwise. The \emph{boundary} $\dM$ of $M$ is the set of all $p\in M$ that are in the domain of some boundary chart $(U,\phi)$ with $\phi(p)\in\partial\R_+^d$. In case $\dM=\emptyset$, $M$ is just a conventional $d$-dimensional topological manifold (without boundary). The interior $\mathring{M}=M\setminus\dM$ of $M$ is an open subset of $M$ and a $d$-dimensional topological manifold without boundary, whereas $\dM$ is a closed subset of $M$ and a $d-1$-dimensional topological manifold without boundary.

    A \emph{$d$-dimensional smooth manifold with boundary} is a $d$-dimensional topological manifold with boundary such that any two charts $(U,\phi)$ and $(V,\psi)$ are smoothly compatible, i.e.\ 
    \begin{equation*}
        \phi\circ\psi^{-1}\colon \psi(U\cap V)\to\phi(U\cap V)
    \end{equation*}
    is a smooth diffeomorphism. Here smoothness of a function $f$ defined on a subset $A$ of $\R^n$ that may not be open means that it is the restriction to $A$ of some smooth function $\tilde f$ defined on an open set $\tilde A\subset\R^n$ containing $A$. A map $f$ between two smooth manifolds $M$ and $N$ is called smooth if for any charts $(U,\phi)$ in $M$ and $(V,\psi)$ in $N$ with $f(U)\subset V$, the composition 
    \begin{equation*}
        \psi\circ f\circ\phi^{-1}\colon \phi(U)\to \psi(V)
    \end{equation*}
    is smooth. The set of smooth functions from $M$ to $N$ is denoted $C^\infty(M,N)$, and we set $C^\infty(M)=C^\infty(M,\R)$. The boundary $\dM$ together with the smooth structure given by the restrictions of the charts of $M$ to $\dM$ is a $(d-1)$-dimensional smooth manifold (without boundary). 
    
    In the following, let $M$ be a $d$-dimensional smooth manifold with boundary.
    As in the case of smooth manifolds without boundary, for any point $p\in M$, the tangent space $\T_pM$ is the vector space of derivations at $p$, i.e.\ linear maps $v\colon C^\infty(M)\to\R$ such that
    \begin{equation*}
        v(fg) = f(p)v(g) + g(p)v(f)\quad\textup{for all }f,g\in C^\infty(M).
    \end{equation*}
    Note that $\T_pM$ is a $d$-dimensional vector space, regardless of whether $p\in\dM$ or $p\in\mathring M$. As usual, the \emph{tangent bundle} $\T M$ is the disjoint union
    \begin{equation*}
        \T M = \coprod_{p\in M}\T_pM
    \end{equation*}
    of the tangent spaces and is again a smooth manifold. A \emph{vector field} on $M$ is a Borel-measurable section of $\T M$, and the space of vector fields on $M$ is denoted by $\Gamma(\T M)$. For $k\in\N_0\cup\{\infty\}$, a \emph{$C^k$-vector field} on $M$ is a $C^k$-section of $\T M$. The space of $C^k$-vector fields is denoted by $\Gamma_{C^k}(\T M)$, and we denote the space of smooth vector fields by $\X(M) = \Gamma_{C^\infty}(\T M)$.

    As in the case without boundary, a \emph{Riemannian metric} $g$ on $M$ is a smooth covariant $2$-tensor field on $M$ that is symmetric and positive definite. This is nothing else than an inner product $g_p$ on each tangent space $\T_pM$ that varies smoothly in $p$. Such a pair $(M,g)$ is called a \emph{$d$-dimensional Riemannian manifold with boundary}. The smooth submanifold $\dM$ of $M$ can be endowed with a Riemannian metric induced by $g$ in the following way. Letting $\iota\colon\dM\to M$ denote the inclusion map and $\diff\iota_p\colon\T_p\dM\to\T_{\iota(p)}M$ its differential at $p\in\dM$, the smooth submanifold $\dM$ of $M$ can be endowed with the \emph{induced Riemannian metric}
    \begin{equation*}
        \tilde g_p(v,w) = g_{\iota(p)}(\diff\iota_p(v),\diff\iota_p(w))\,.
    \end{equation*}
    Unwinding the definition, $\diff\iota_p(v)$ is simply the derivation that acts as $\diff\iota_p(v)(f) = v(f|_\dM)$ for all $f\in C^\infty(M)$. For sake of simplicity, we usually identify $\T_p\dM$ with its image $\diff\iota_p(\T_p\dM)$ in $\T_{\iota(p)}M=\T_pM$ and again denote $\tilde g$ by $g$. Thus the induced metric is just the restriction of $g$ to vectors tangent to $\dM$. Once a Riemannian metric $g$ is fixed, for any $p\in M$ and $v,w\in\T_pM$ we also write $\langle v,w\rangle = g(v,w)$ and $|v| = \langle v,v\rangle^{1/2}$ for the inner product and norm on $\T_pM$ given by $g$. 
    
    In the following, let $(M,g)$ be a $d$-dimensional Riemannian manifold with boundary.
    The normal space $N_p\dM$ at $p\in\dM$ is the orthogonal complement of $\T_p\dM$ in $\T_pM$ with respect to the inner product $\langle\cdot,\cdot\rangle$, and the normal bundle is the disjoint union
    \begin{equation*}
        \Ncal\dM = \coprod_{p\in\dM}\Ncal_p\dM
    \end{equation*}
    of the normal spaces. The space of smooth sections of $\Ncal\dM$ is denoted by $\mathfrak{N}(\dM)$. There exists a unique smooth outward-pointing unit normal vector field $N\in\mathfrak{N}(\dM)$ along all of $\dM$, see \cite[Prop. 2.17]{Lee2018Riemannian}.

    Since the product of two smooth manifolds with boundary is not necessarily again a smooth manifold with boundary, it is necessary to introduce \emph{smooth manifolds with corners}, see \cite[Chapter 16]{Lee2013Smooth}. They are defined analogously to smooth manifolds with boundaries, with the difference being that every point has a neighbourhood that is homeomorphic to a subset of $\overline{\R}_+^d = \{(x_1,\dots,x_d)\in\R^d\colon x_1\geq0,\dots,x_d\geq0\}$. In particular, any smooth manifold with boundary is a smooth manifold with corners. 
    The corner points of $\overline{\R}_+^d$ are all the points at which more than one coordinate vanishes, and the corner points of a smooth manifold with corners are all the points $p$ that are in the domain of some corner chart mapping $p$ to a corner point of $\overline{\R}_+^d$. The tangent bundle, vector fields and Riemannian metrics are defined analogously to the case with boundary. The product of a finite number of smooth manifolds with corners is again a smooth manifold with corners, and removing the corner points from a smooth manifold with corners yields a smooth manifold with boundary.

    \subsection{The Hessian and the second fundamental form}\label{ssec:HessianSFF}
    
    Let $(M,g)$ be a $d$-dimensional Riemannian manifold with boundary. The \emph{Riemannian gradient} of a smooth function $f\in C^\infty(M)$ is the smooth vector field $\nabla f$ defined by
    \begin{equation*}
        \langle \nabla f,v\rangle = \diff f(v)\quad\textup{for all }v\in \T M\,,
    \end{equation*}
    where $\diff f$ is the differential of $f$.

    The Levi-Civita connection
    \begin{align*}
        \nabla\colon\X(M)\times\X(M)\to\X(M),\quad(X,Y)\mapsto\nabla_XY
    \end{align*}
    defines the \emph{covariant derivative} of a vector field $Y$ in direction of the vector field $X$.
    Since $\nabla_XY|_p$ only depends on $X$ through the value of $X$ at $p$, the derivative $\nabla_vY$ of $Y$ at $p$ in direction $v$ is defined for any $v\in\T_pM$. Furthermore, $\nabla$ uniquely determines a connection in each tensor bundle $\T^{k,l}\T M$ that agrees with the usual Levi-Civita connection on $\T^{1,0}\T M = \T M$ and the Riemannian gradient on $\T^{0,0}=M\times\R$ in the sense that $\nabla_Xf = \langle \nabla f,X\rangle$ for all $f\in C^\infty(M)$ and $X\in\X(M)$, see \cite[Prop. 4.15]{Lee2018Riemannian}.
    
    The \emph{Hessian} of a smooth function $f\in C^\infty(M)$ is the symmetric $2$-tensor $\nabla^2f$. Its action on vector fields $X,Y\in\X(M)$ is given by
    \begin{equation*}
        \nabla^2f(X,Y)=\nabla^2_{X,Y}f = \nabla_X(\nabla_Yf) - \nabla_{\nabla_XY}f\,.
    \end{equation*}

    For any $p\in\dM$ and $Z\in\T_pM$ denote by $Z_p^\top$ and $Z_p^\bot$ the projections of $Z_p$ onto $\T_p\dM$ and $\Ncal_p\dM$, respectively. Extending vector fields $X$ and $Y$ on $\dM$ arbitrarily to vector fields on $M$, we can decompose
    \begin{equation*}
        \nabla_XY = (\nabla_XY)^\top + (\nabla_XY)^\bot\,.
    \end{equation*}
    The Gauß formula, see \cite[Theorem 8.2]{Lee2018Riemannian}, identifies $(\nabla_XY)^\top$ as the covariant derivative $\nablad_XY$ with respect to the Levi-Civita connection $\nablad$ on $\T\dM$ of the induced metric $g$ on $\dM$. This leads to the \emph{second fundamental form} $\sff\colon\X(\dM)\times\X(\dM)\to\mathfrak{N}(\dM)$ defined by
    \begin{equation}\label{eq:defh}
        \sff(X,Y) = (\nabla_XY)^\bot = \nabla_XY-\nablad_XY\,.
    \end{equation}
    The \emph{scalar second fundamental form} $h\colon\X(\dM)\times\X(\dM)\to C^\infty(\dM)$ is then defined by
    \begin{equation*}
        h(X,Y) = \langle\sff(X,Y),N\rangle = \langle \nabla_XY,N\rangle\,.
    \end{equation*}
    Both $\sff$ and $h$ are bilinear over $C^\infty(M)$ and symmetric, and the values of $h(X,Y)$ and $\sff(X,Y)$ at $p\in\dM$ only depend on the values of $X$ and $Y$ at $p$. The \emph{mean curvature} $H$ is 
    \begin{equation*}
        H = \tr(h)\,,
    \end{equation*}
    where $\tr$ is the trace operator.

    \subsection{The Laplace-Beltrami operator}\label{ssec:LaplaceBeltrami}
    Let $(M,g)$ be a $d$-dimensional Riemannian manifold with corners. The \emph{Riemannian volume measure} $\nu_M$ is the Borel measure on $M$ given by
    \begin{equation*}
        \diff\nuM = \sqrt{\det{g}}\,\lvert\diff x^1\wedge\dots\wedge \diff x^d\rvert
    \end{equation*}
    in any local coordinates $(x^1,\dots,x^d)$.
    The \emph{divergence} of a vector field $X\in\Gamma_{C^1}(\T M)$ is $\divg(X) = \tr(\nabla X)$. In local coordinates $X=\sum_{i=1}^dX^i\frac{\partial}{\partial x^i}$ it takes the form
    \begin{equation*}
        \divg(X) = \frac{1}{\sqrt{\det{g}}}\sum_{i=1}^d\frac{\partial}{\partial x^i}\left(X^i\sqrt{\det{g}}\right)\,.
    \end{equation*}
    It satisfies the integration-by-parts identity
    \begin{equation*}
        \int_M \langle\nabla f,X\rangle\diff\nuM = \int_\dM f\langle X,N\rangle\diff\nudM - \int_M(f\divg X)\diff\nuM\,,
    \end{equation*}
    where $\nudM$ is the Riemannian volume measure on $(\dM,g)$. The \emph{Laplace-Beltrami operator} is the linear operator $\Delta\colon C^\infty(M)\to C^\infty(M)$ defined by
    \begin{equation*}
        \Delta u = \divg(\nabla u) = \tr(\nabla^2u)\,.
    \end{equation*}
    In coordinates it takes the form
    \begin{equation*}
        \Delta u = \frac{1}{\sqrt{\det{g}}}\sum_{i,j=1}^d\frac{\partial}{\partial x^i}\left((g^{-1})_{ij}\sqrt{\det{g}}\frac{\partial u}{\partial x^j}\right)\,.
    \end{equation*}
    If $M\subset\R^d$ with the Euclidean metric, these definitions reduce to the usual divergence and Laplacian.
    The definition of the Laplace-Beltrami operator $\Delta$ immediately implies the integration-by-parts identity
    \begin{align*}
        \int_M \phi \Delta\psi\diff\nuM = \int_\dM \phi \langle\nabla\psi,N\rangle\diff\nudM - \int_M\langle\nabla\phi,\nabla\psi\rangle\diff\nuM
    \end{align*}
    for all smooth functions $\phi,\psi\in C^\infty(M)$.

    \subsection{A generalisation of the Reilly formula}\label{ssec:BochnerReilly}
    
    Let $(M,g)$ be a $d$-dimensional Riemannian manifold with boundary. In this section, we prove \Cref{lem:GenReilly}.
    
    \GenReilly*

    \begin{proof}
        Let $u\in C^\infty(M)$. The integration-by-parts identity \eqref{eq:intbypartsL} yields
        \begin{equation*}
            \frac{1}{2}\int_M L|\nabla u|^2\diff\mu = \frac{1}{2}\int_\dM\langle\nabla|\nabla u|^2,N\rangle\diff\mudM
        \end{equation*}
        and
        \begin{equation*}
            \int_M\langle\nabla u,\nabla L u\rangle\diff\mu = -\int_M (Lu)^2\diff\mu + \int_\dM Lu\langle\nabla u,N\rangle\diff\mudM\,.
        \end{equation*}
        Thus integrating the Bochner identity \eqref{eq:bochner2} shows
        \begin{equation*}
            \int_M\left(|\nabla^2u|^2-(Lu)^2+(\Ric+\nabla^2U)(\nabla u,\nabla u)\right)\diff\mu=\int_\dM\left\langle\frac{1}{2}\nabla|\nabla u|^2-\Delta u\nabla u,N\right\rangle\diff\mudM\,.
        \end{equation*}
        By splitting $\nabla u = \nabla^\dM u + (\nabla u)^\bot $, one sees that
        \begin{align*}
            \frac{1}{2}\langle\nabla|\nabla u|^2,N\rangle &= \langle\nabla_{\nabla u}\nabla u,N\rangle\\
            &=\langle\nabla_{\nablad u}\nablad u,N\rangle + \langle\nabla_{\nablad u}(\nabla u)^\bot ,N\rangle + \langle\nabla_{(\nabla u)^\bot}\nabla u,N\rangle\,.
        \end{align*}
        The first summand is just the scalar second fundamental form $h(\nablad u,\nablad u)$. For the second summand, note that any vector field $X$ defined on the boundary $\dM$ satisfies
        \begin{equation*}
            \nabla_{X^\top}X^\bot = \langle X,N\rangle\nabla_{X^\top}N + \langle X^\top,\nabla\langle X,N\rangle\rangle N\,,
        \end{equation*}
        so that
        \begin{equation*}
            \langle\nabla_{X^\top}X^\bot,N\rangle = \langle X^\top,\nabla\langle X,N\rangle\rangle = \langle X^\top,\nablad\langle X,N\rangle\rangle\,.
        \end{equation*}
        This yields
        \begin{equation*}
            \langle\nabla_{\nablad u}(\nabla u)^\bot,N\rangle = \langle\nablad u,\nablad\langle\nabla u,N\rangle\rangle
        \end{equation*}
        and thus
        \begin{equation*}
            \frac{1}{2}\langle\nabla|\nabla u|^2,N\rangle = h(\nablad u,\nablad u) +  \langle\nablad u,\nablad\langle\nabla u,N\rangle\rangle + \langle\nabla_{(\nabla u)^\bot}\nabla u,N\rangle\,.
        \end{equation*}
        Furthermore, in an orthonormal boundary frame $(E_1,\dots,E_{d-1},N)$, one has
        \begin{align*}
            \Delta u &= \tr(\nabla^2 u) = \langle\nabla_N\nabla u,N\rangle + \sum_{i=1}^{d-1}\langle\nabla_{E_i}\nabla u,E_i\rangle\\
            &=\langle\nabla_N\nabla u,N\rangle + \sum_{i=1}^{d-1}\left(\langle\nabla_{E_i}\nablad u,E_i\rangle+\langle\nabla_{E_i}(\nabla u)^\bot,E_i\rangle\right)\\
            &=\langle\nabla_N\nabla u,N\rangle + \Deltad u + \sum_{i=1}^{d-1}\langle\nabla_{E_i}(\nabla u)^\bot,E_i\rangle
        \end{align*}
        on $\dM$.
        For any tangential vector field $X$ defined on the boundary $\dM$, one has
        \begin{align*}
            0 &= X\langle N,X\rangle = \langle \nabla_XN,X\rangle + \langle N,\nabla_XX\rangle\\
            &=\langle \nabla_XN,X\rangle + \langle N,\sff(X,X)+\nablad_XX\rangle\\
            &=\langle \nabla_XN,X\rangle + h(X,X)\,.
        \end{align*}
        Hence $\langle\nabla_{E_i}(\nabla u)^\bot,E_i\rangle = -\langle\nabla u,N\rangle h(E_i,E_i)$ and thus
        \begin{align*}
            Lu &= \Delta u-\langle\nabla U,\nabla u\rangle\\
            &= \langle\nabla_N\nabla u,N\rangle + \Deltad u - \langle\nabla u,N\rangle H-\langle\nablad U,\nablad u\rangle-\langle(\nabla U)^\bot,(\nabla u)^\bot\rangle\\
            &=\langle\nabla_N\nabla u,N\rangle + \Ld u -(H+\langle\nabla U,N\rangle)\langle\nabla u,N\rangle
        \end{align*}
        on the boundary $\dM$.

        Putting things together and integrating by parts on $\dM$ finally yields
        \begin{align*}
            \MoveEqLeft\int_\dM\left\langle\frac{1}{2}\nabla|\nabla u|^2-L u\nabla u,N\right\rangle\diff\mudM\\
            &=\int_\dM\left( h(\nablad u,\nablad u) +  \langle\nablad u,\nablad\langle\nabla u,N\rangle\rangle + \langle\nabla_{(\nabla u)^\bot}\nabla u,N\rangle \right)\diff\mudM\\
            &\quad-\int_\dM \left(\langle\nabla_N\nabla u,N\rangle + \Ld u - (H+\partial_NU)\langle\nabla u,N\rangle\right)\langle\nabla u,N\rangle\diff\mudM\\
            &=\int_\dM \left(h(\nablad u,\nablad u)-2(\partial_N u)\Ld u + (H+\partial_NU)(\partial_Nu)^2\right)\diff\mudM\,,
        \end{align*}
        completing the proof.
    \end{proof}

\section*{Statements and Declarations}
    \noindent\textbf{Funding.}\hspace{2ex}
    Gef\"ordert durch die Deutsche Forschungsgemeinschaft (DFG) im Rahmen der Exzellenzstrategie des Bundes und der L\"ander -- GZ2047/1, Projekt-ID 390685813.\\
    The authors were funded by the Deutsche Forschungsgemeinschaft (DFG, German Research Foundation) under Germany's Excellence Strategy  -- GZ 2047/1, Project-ID 390685813.\smallskip\\
    \textbf{Acknowledgements.}\hspace{2ex}
    The authors would like to thank Lihan Wang, Arnaud Guillin and L\'eo Hahn for many helpful discussions.
    
\printbibliography

\end{document}